\documentclass{amsart}
\usepackage{graphicx}
\usepackage[export]{adjustbox}
\graphicspath{ {./images/} }
\usepackage{amssymb}
\usepackage{amsmath}
\usepackage{float}
\usepackage{caption}
\usepackage[dvipsnames]{xcolor} 
\usepackage{tikz-cd}
\usepackage{enumitem}

\newtheorem{theorem}{Theorem}[section]
\newtheorem{claim}[theorem]{Claim}

\newtheorem{lemma}[theorem]{Lemma}

\newtheorem{proposition}[theorem]{Proposition}
\newtheorem{corollary}[theorem]{Corollary}

\newtheorem{assumption}[theorem]{Assumption}

\theoremstyle{definition}
\newtheorem{definition}[theorem]{Definition}
\newtheorem{example}[theorem]{Example}

\theoremstyle{remark}
\newtheorem{remark}[theorem]{Remark}

\DeclareMathOperator{\dom}{dom}

\DeclareMathOperator{\Succ}{Succ}
\DeclareMathOperator{\rng}{rng}
\DeclareMathOperator{\mc}{mc}
\DeclareMathOperator{\OB}{OB}
\DeclareMathOperator{\tp}{top}
\DeclareMathOperator{\Lev}{Lev}
\DeclareMathOperator{\crit}{crit}
\DeclareMathOperator{\Add}{Add}
\DeclareMathOperator{\stem}{stem}
\newcommand{\cf}{{\rm cf}}
\def\Ult{{\rm Ult}}

\title{Extender-based Magidor-Radin forcings without top extenders}
\author{Moti Gitik}
\address{School of Mathematical Sciences, Tel Aviv University, Tel Aviv-Yafo,  Tel Aviv, Israel, 6997801}
\email{gitik@tauex.tau.ac.il}
\author{Sittinon Jirattikansakul}
\address{School of Mathematical Sciences, Tel Aviv University, Tel Aviv-Yafo,  Tel Aviv, Israel, 6997801}
\email{jir.sittinon@gmail.com}
\thanks{The authors were partially supported by ISF grant No. 882/22.}
\thanks{We are grateful to Carmi Merimovich for many helpful discussions on the subject. }

\begin{document}
\maketitle

\begin{abstract}
    Continuing \cite{GitJir22}, we develop a version of Extender-based Magidor-Radin forcing where there are no extenders on the top ordinal. As an application, we provide another approach to obtain a failure of SCH on a club subset of an inaccessible cardinal, and a model where the cardinal arithmetic behaviors are different on stationary classes, whose union is the club, is provided. The cardinals and the cofinalities outside the clubs are not affected by the forcings. 
\end{abstract}

\section{Introduction}

The present work continues \cite{GitJir22} and develops Extender-based Magidor-Radin forcings without top extenders.
The main new issue here is to deal with Cohen parts of Extenders Based forcings.
New ideas involving a substantial use of names will be applied for  this.

As an application, we give new proofs of results of \cite{GITIKmerimovich2006}, where the power set function behaves differently on stationary classes.
An advantage of the present approach is that fewer cardinals and cofinalities are affected by the forcing.

The organization of the paper is the following.
In Section \ref{bas} we introduce all basic ingredients we need to develop the forcing.
From Section \ref{firfew} to Section \ref{mainforcing}, we develop the forcing in which a club class of cardinals $\alpha$ with $2^\alpha=\alpha^{++}$.
The forcing for building a club class of cardinals is built from approximated forcings, which will be built by recursion. The basic cases are constructed in Section \ref{firfew}. In Section \ref{theindsch} we state all the properties we need to be true, and show that the forcings in the basic cases satisfy the properties. Then the construction proceeds in Section \ref{belthefir}, Section \ref{atfir}, and Section \ref{generallevel}. The main forcing will then be introduced in \ref{mainforcing}. Lastly, in Section \ref{getdifcar}, we sketch a generalization of the forcing to get different cardinal behaviors on different stationary classes. 

Although the version of Extender-based forcing and the Extender-based Magidor-Radin forcing we will be using looks different from \cite{Mer11}, A familiarity of the Extender-Based Magidor-Radin forcings will accommodate the readers.

Conventions: Without mentioning, we assume that every forcing has the weakest element $1$. $p \leq q$ means $p$ is stronger than $q$. When possible, every name in this paper will be in the canonical form. Most of the time, we omit the check symbol when we discuss the check names. For sets $A$ and $B$, $A \sqcup B$ just means $A \cup B$ where $A \cap B=\emptyset$. If $f$ is a function and $d$ is a set, define $f \restriction d$ as $f \restriction [d\cap \dom(f)]$. If $f$ and $g$ are functions, $f \circ g$ is a function whose domain is $\{x \in \dom(g) \mid g(x) \in \dom(f)\}$ and $f \circ g(x)=f(g(x))$. Throughout the paper, the forcing at level $\rho$, denoted $P_\rho$, will be defined. We often abbreviate $\leq_{P_\rho}$ by $\leq_\rho$ and $\Vdash_{P_\rho}$ by $\Vdash_\rho$. If $\vec{x}=\langle x_{\alpha,\beta} \rangle$ is a sequence indexed by pairs of ordinals, we define $$\vec{x} \restriction (\alpha,\beta)=\langle x_{\alpha^\prime,\beta^\prime} \mid \alpha^\prime<\alpha \text{ or }(\alpha^\prime=\alpha \text{ and }\beta^\prime<\beta) \rangle,$$ and
$$\vec{x} \restriction \alpha= \vec{x} \restriction (\alpha,0).$$

\section{Basic preparation}
\label{bas}
From now until Section \ref{mainforcing}, we have the following hypotheses.

\begin{assumption}
\label{iniass}
GCH holds. $\kappa$ is a strongly inaccessible cardinal. There is a function $\circ: \kappa \to \kappa$ and $\vec{E}=\langle E(\alpha,\beta) \mid \alpha<\kappa, \beta<\circ(\alpha) \rangle$ such that

\begin{enumerate}

    \item $E(\alpha,\beta)$ is an $(\alpha,\alpha^{++})$-extender, which means that if $$j_{\alpha,\beta}: V \to \Ult(V,E(\alpha,\beta))=:M_{\alpha,\beta}$$ is the ultrapower map, then $\crit(j_{\alpha,\beta})=\alpha$, and $M_{\alpha,\beta}$ computes cardinals correctly up to an including $\alpha^{++}$.

    \item $\vec{E}$ is {\em coherent}, namely $$j_{\alpha,\beta}(\vec{E}) \restriction (\alpha+1)=\vec{E} \restriction (\alpha,\beta).$$

    \item for all $\alpha$, $\circ(\alpha)<\alpha$.

    \item For every $\gamma<\kappa$, the collection 
    $$\{\alpha<\kappa \mid \circ(\alpha) \geq \gamma\}$$
is stationary.    
\end{enumerate}

\end{assumption}

\begin{definition}
Let $\alpha<\kappa$. We say that $d$ is a {\em $\alpha$-domain} if $d \in [\alpha^{++} \setminus \alpha]^{\leq \alpha}$ and $\alpha \in d$.
Define $C(\alpha^+,\alpha^{++})$ as the collection of functions $f$ such that $\dom(f)$ is a $\alpha$-domain $d$, and $\rng(f) \subseteq \alpha$. Define the ordering in $C(\alpha^+,\alpha^{++})$ by $f \leq g$ iff $f \supseteq g$.
\end{definition}

Note that $C(\alpha^+,\alpha^{++})$ is isomorphic to $\Add(\alpha^+,\alpha^{++})$, the forcing which adds $\alpha^{++}$ Cohen subsets of $\alpha^+$.

\begin{remark}
\label{densegrounddomain}
If $|P|\leq \alpha$ and $\dot{C}(\alpha^+,\alpha^{++})$ is a $P$-name of the forcing interpreted in the extension, then $$\Vdash_P``\{\dot{f} \in \dot{C}(\alpha^+,\alpha^{++}) \mid \dom(\dot{f})=\check{d}, d \in V\} \text{ is dense}".$$
We identify such and $\dot{f}$ by $f$ with $\dom(f)=d$, and for $\alpha \in \dom(f)$, $f(\alpha)$ is a $P$-name of an ordinal below $\alpha$.
\end{remark}

Until the end of this section, fix $\alpha$ with $\circ(\alpha)>0$ and $\beta<\circ(\alpha)$. We introduce some definitions and facts which will be used since Section \ref{generallevel}. Fix an $\alpha$-domain $d$.

\begin{itemize}
    \item Define $\mc_{\alpha,\beta}(d)=\{(j_{\alpha,\beta}(\xi),\xi) \mid \xi \in d\}$. 
    \item Define $E_{\alpha,\beta}(d)$ by $X \in E_{\alpha,\beta}(d)$ iff $\mc_{\alpha,\beta}(d) \in j_{\alpha,\beta}(X)$. Then $E_{\alpha,\beta}(d)$ concentrates on the collection {\em $\OB_{\alpha,\beta}(d)$} of {\em $(\alpha,\beta)$-$d$-objects}, which are functions $\mu$ such that
    \begin{itemize}
        \item $\alpha \in \dom(\mu) \subseteq d, \rng(\mu) \subseteq \alpha$ (in fact, we can assume that $\rng(\mu) \subseteq \mu(\alpha)^{++}$).
        
        (The reason is that $\dom(\mc_{\alpha,\beta}(d))=j_{\alpha,\beta}[d]\subseteq j_{\alpha,\beta}(d)$, $j_{\alpha,\beta}(\alpha) \in j_{\alpha,\beta}[d]$, $\rng(\mc_{\alpha,\beta}(d))=d \subseteq \alpha^{++}=\mc_{\alpha,\beta}(d)(j_{\alpha,\beta}(\alpha))^{++})$.
        \item $\circ(\mu(\alpha))=\beta$, in particular, $\mu(\alpha)$ is strongly inaccessible,  $|\dom(\mu)| \leq \mu(\alpha)^{++}$, and $\mu$ is order-preserving.

        (The reason is that $j_{\alpha,\beta}(\circ)(\alpha)^{M_{\alpha,\beta}}=\beta$, $\alpha$ is inaccessible, $|\dom(\mc_{\alpha,\beta}(d))|=|d|\leq \alpha^{++}$, and $\mc_{\alpha,\beta}$ is order-preserving.)
    \end{itemize}
    \item Let $X_\nu \in E_{\alpha,\beta}(d)$ for $\nu<\alpha$. Define the {\em diagonal intersection} $$\Delta_{\nu<\alpha}X_\nu=\{\mu \in \OB_{\alpha,\beta}(d) \mid \forall \nu<\mu(\alpha)(\mu \in X_\nu)\}.$$
    Then $\Delta_{\nu<\alpha}X_\nu \in E_{\alpha,\beta}(d)$.

    \item The measure $E_{\alpha,\beta}(\{\alpha\})$ is normal, and is isomorphic to  $E_{\alpha,\beta}(\alpha)$, which is defined by $X \in E_{\alpha,\beta}(\alpha)$ iff $\alpha \in j_{\alpha,\beta}(X)$.
    
    \item if $d^\prime \supseteq d$ is an $\alpha$-domain, there is an {\em associated  projection} from $E_{\alpha,\beta}(d^\prime)$ to $E_{\alpha,\beta}(d)$ induced by the map $\pi_{d^\prime,d}:\OB_{\alpha,\beta}(d^\prime) \to \OB_{\alpha,\beta}(d)$ defined by $\pi_{d^\prime,d}(\mu)=\mu \restriction d$ (i.e. $\mu \restriction (d \cap \dom(\mu)$). In particular, there is a projection from $E_{\alpha,\beta}(d)$ to $E_{\alpha,\beta}(\{\alpha\})$.

    \item Similar as in the proof of Lemma 2 
 \cite{Jir22}, there is a measure-one set $\mathcal{B}_d \in E_{\alpha,\beta}(d)$ such that for every $\nu<\alpha$, $\{\mu \in \OB_{\alpha,\beta}(d) \mid \mu(\alpha)=\nu\} \leq \nu^{++}$. We will assume that for every $A \in E_{\alpha,\beta}(d)$, $A \subseteq \mathcal{B}_d$.

\end{itemize}

We now no longer fix $\beta$, but still fix $\alpha$ and $d$.

\begin{itemize}

    \item $\mu$ is an {\em $\alpha$-$d$-object} if $\mu$ is an $(\alpha,\beta)$-$d$-object for some $\beta<\circ(\alpha)$.
    Denote the collection of $\alpha$-$d$-object by $\OB_\alpha(d)$. For each pair of $\alpha$-$d$-objects $\mu$ and $\tau$, define $\mu<\tau$ if $\dom(\mu) \subseteq \dom(\tau)$ and  for $\gamma \in \dom(\mu)$, $\mu(\gamma)<\tau(\gamma)$.

    \item Define $X \in \vec{E}_\alpha(d)$ iff $X$ can be written as $X=\cup_{\beta<\circ(\alpha)} X_\beta$ where $X_\beta \in E_{\alpha,\beta}(d)$.
    Note that for each $\alpha$-$d$-object $\mu$, $\{\tau \in \OB_\alpha(d) \mid \mu<\tau\} \in \vec{E}_\alpha(d)$.

    \item Note that for each $\alpha$-$d$-object $\tau$, $\{\mu \mid \tau<\mu\} \in \vec{E}_\alpha(d)$. 
    \item For each $X \in \vec{E}_\alpha(d)$, $X$ can be written as a disjoin union of $X_\beta$, $\beta<\alpha$, where $X_\beta \in E_{\alpha,\beta}(d)$ and for each $\mu \in X_\beta$, $\circ(\mu(\alpha))=\beta$.
    \item Let $X_\nu \in \vec{E}_\alpha(d)$ for $\nu<\alpha$. The {\em diagonal intersection}
    $$\Delta_{\nu<\alpha}X_\nu=\{\mu \in \OB_\alpha(d) \mid \forall \nu<\mu(\alpha)(\mu \in X_\nu)\}$$
    is in $\vec{E}_\alpha(d)$.

    \item If $\mu<\tau$, we define $\mu \downarrow \tau=\mu \circ \tau^{-1}$, which is the function whose domain is $\tau[\dom(\mu)]$ and for $\gamma \in \dom(\mu)$, $(\mu \downarrow \tau)(\tau(\gamma))=\mu(\gamma)$. Since $\tau$ is order-preserving, we have that $\mu \downarrow \tau$ is well-defined. 
    
    \item If $X$ is a set of $\alpha$-$d$-object and $\tau \in \OB_\alpha(d)$, define $X \downarrow \tau=\{\mu \downarrow \tau \mid \mu<\tau, \circ(\mu(\alpha))<\circ(\tau(\alpha))\}$. By the coherence of the extenders, we also assume that every $X \in \vec{E}_\alpha(d)$ is {\em coherent}, i.e. for every $\tau \in X$, $X \downarrow \tau \in \vec{E}_{\tau(\alpha)}(\tau[d \cap \dom(\tau)])$. 
    
    \item Let $\vec{\mu}=\langle \mu_0 ,\cdots, \mu_{n-1} \rangle$ be an increasing sequence of $\alpha$-$d$-objects, define $\vec{\mu}(\alpha)=\mu_{n-1}(\alpha)$, which is just an inaccessible cardinal below $\alpha$. Also write $\dom(\vec{\mu})=\dom(\mu_{n-1})$. Also, if $\mu_{n-1}<\tau$, we define $\vec{\mu} \downarrow \tau=\langle \mu_0 \downarrow \tau, \cdots, \mu_{n-1} \downarrow \tau \rangle$.

    \item $A$ is an $\alpha$-$d$-tree if $A$ consists of nonempty finite increasing sequences of $\alpha$-$d$-objects, and $A$ has the following descriptions:
    \begin{itemize}
        \item $\vec{\mu} \leq_A \vec{\tau}$ iff $\vec{\mu} \sqsubseteq \vec{\tau}$ ($\vec{\mu}$ is an initial segment of $\vec{\tau}$).
        \item $\Lev_n(A)$ is the collection of $\langle \mu_0,\cdots, \mu_n \rangle$ in $A$, so they have lengths $n+1$.
        \item We require that $\Lev_0(A) \in \vec{E}_\alpha(d)$.
        \item For $\vec{\mu} \in A$, define $\Succ_A(\vec{\mu})=\{\tau \mid \vec{\mu} {}^\frown \langle \tau \rangle \in A\}$. We require that $\Succ_A(\vec{\mu}) \in \vec{E}_\alpha(d)$.
    \end{itemize}

    \item If $A$ is an $\alpha$-$d$-tree and $\mu \in \Lev_0(A)$, define $A_{\langle \mu \rangle}=\{ \vec{\tau} \mid \langle \mu \rangle{}^\frown \vec{\tau} \in A\}$, and we recursively define $A_{\langle \mu_0 ,\cdots, \mu_n \rangle}=(A_{\langle \mu_0 \cdots, \mu_{n-1} \rangle})_{\langle \mu_n \rangle}$.

    \item Fix $d^\prime \subseteq d$ an $\alpha$-domain and $\vec{\mu}=\langle \mu_0,\cdots, \mu_{n-1} \rangle$ is a finite increasing sequence of $\alpha$-$d$-objects, define $\vec{\mu} \restriction d^\prime=\langle \mu_0 \restriction d,\cdots, \mu_{n-1} \restriction d^\prime \rangle$. If we assume that $A$ is an $\alpha$-$d$-tree, define $A \restriction d^\prime=\{\vec{\mu} \restriction d^\prime \mid \vec{\mu} \in A\}$. Then $A \restriction d^\prime$ is an $\alpha$-$d^\prime$-tree.

    \item If $d^\prime \supseteq d$ is an $\alpha$-domain, and $A$ is an $\alpha$-$d$-tree, the {\em pullback of $A$ to $d^\prime$}, is $\{\vec{\mu} \in [\OB_\alpha(d^\prime)]^{<\omega} \mid \vec{\mu}$ is increasing and $\vec{\mu} \restriction d \in A\}$. Note that the pullback is an $\alpha$-$d^\prime$-tree.

    \item A tree $A$ is {\em generated by} $B \in \vec{E}_\alpha(d)$ if $\Lev_0(A)=B$, and for $\vec{\mu}=\langle \mu_0, \cdots, \mu_{n-1} \rangle \in A$, $\Succ_A(\vec{\mu})=\{\tau \in B \mid \mu_{n-1}<\tau\}$. Such a tree is an $\alpha$-$d$-tree. Furthermore, every $\alpha$-$d$-tree $A$ has a sub $\alpha$-$d$-tree which is generated by some $B \in \vec{E}_\alpha(d)$: for each $\nu<\alpha$, let $X_\nu=\cap_{\vec{\mu} \in T,\vec{\mu}(\alpha)\leq \nu}\Succ_A(\vec{\mu})$, and $B=\Delta_\nu X_\nu$. We assume that every $d$-tree $A$ is generated by some $B \subseteq \mathcal{B}_d$.

    \item We write $A(\alpha)=\{\vec{\mu}(\alpha) \mid \vec{\mu} \in A\}$. If $A$ is generated by $B$, then $A(\alpha)=B(\alpha)=\{\mu(\alpha) \mid \mu \in B\}$.

    \item If $A$ is an $\alpha$-$d$-tree and $\tau$ is an object, define $A \downarrow \tau=\{\vec{\mu} \downarrow \tau \mid \forall i(\mu_i<\tau \text{ and } \circ(\mu_i(\alpha))<\circ(\tau(\alpha)))\}$. By the coherence, assume that for each $\tau$, $A \downarrow \tau$ is an $\tau(\alpha)$-$\tau[d \cap \dom(\tau)]$-tree, with respects to $\vec{E}_{\tau(\alpha)}(\tau[d \cap \dom(\tau)])$.   
    \end{itemize} 

\begin{remark}
\label{ppoitre}
    For every $d$-tree $A$ and $\nu<\alpha$, we assume that $\{\vec{\mu} \in A \mid \vec{\mu}(\alpha)=\mu_{|\vec{\mu}|-1}(\alpha)=\nu\}$ has size at most $\nu^{++}$.
\end{remark}    

\section{The first few levels}
\label{firfew}
We consider the forcings at the first $\omega$ inaccessible cardinals, so, the extenders are not involved. We first analyze just for the first few inaccessible cardinals concretely, which will be served as the first few basic cases for our induction scheme for the forcings in the general levels, which will be listed later in Proposition \ref{indsch}.

\subsection{The first inaccessible cardinal}
\label{firlvl}
Let $\alpha_0$ be the least inaccessible cardinal. The following describe the scenario at the level $\alpha_0$.

\begin{itemize}

    \item The forcing $P_{\alpha_0}$ consists of $\langle f \rangle$ where $f \in C(\alpha_0^+,\alpha_0^{++})$. For $\langle f \rangle, \langle g \rangle \in P_{\alpha_0}$, define $\langle f \rangle \leq_{\alpha_0} \langle g \rangle$ iff $f \leq_{\alpha_0}^* g$ iff $f \supseteq g$.

    \item Let $\dot{C}_{\alpha_0}$ be a $P_{\alpha_0}$-name for the set $\{\alpha_0\}$.

    \item Let $\dot{P}_{\alpha_0/\alpha_0}$ be a $P_{\alpha_0}$-name of the trivial forcing  $(\{\emptyset\},\leq,\leq^*)$. We write $P_{\alpha_0}[G]=\dot{P}_{\alpha_0/\alpha_0}[G]$.

    \item In $V^{P_{\alpha_0}}$, let $\dot{C}_{\alpha_0/\alpha_0}$ be a $\dot{P}_{\alpha_0/\alpha_0}$-name of the empty set.

\end{itemize}

The forcing at the first inaccessible cardinal has nothing particularly interesting. The name $\dot{C}_{\alpha_0}$ will be served as the initial approximation of the final club where GCH fails at its limit points.
The quotient forcing like $\dot{P}_{\alpha_0/\alpha_0}$ will show its importance later. $\dot{C}_{\alpha_0/\alpha_0}$ will also be considered for an approximation of the final club.
It will be more suggestive to write $\dot{P}_{\check{\alpha}_0/\alpha_0}$ since in general, the ordinal which appears for the numerator, like $\check{\alpha}_0$, may be a non-trivial name of an ordinal. Since this is a check name, we omit the check symbol.
A trivial remark is that forcing $P_{\alpha_0}*\dot{P}_{\alpha_0/\alpha_0}$ is equivalent to $P_{\alpha_0}$.

\subsection{The second inaccessible cardinal}
Let $\alpha_0<\alpha_1$ be the first two inaccessible cardinals.

\begin{definition}
\label{twolvl}
The forcing $P_{\alpha_1}$ consists of two kinds of conditions (apart from the weakest condition). Conditions of different kinds are not compatible.
\begin{enumerate}
    \item The first kind consists of $\langle f \rangle$ in $C(\alpha_1^+,\alpha_1^{++})$. For $\langle f \rangle$ and $\langle g \rangle$ which are of first kind, define $\langle f \rangle \leq_{\alpha_1} \langle g \rangle$ iff $\langle f \rangle \leq_{\alpha_1}^* \langle g \rangle$ iff $f \supseteq g$. 

    \item The second kind consists of $p=(\langle f_0 \rangle,\langle \dot{P}_{\dot{\xi}/\alpha_0},\dot{q}_0\rangle) {}^\frown \langle f_1 \rangle$, where
    \begin{itemize}
        \item $f_0 \in C(\alpha_0^+,\alpha_0^{++})$.
        \item $\Vdash_{\alpha_0}``\leq \alpha_0 \leq \dot{\xi}<\alpha_1$ is strongly inaccessible" (in this case, we can assume that $\dot{\xi}$ is $\alpha_0$, or more formally, $\check{\alpha}_0$).
        \item $\Vdash_{\alpha_0} ``\dot{q}_0 \in \dot{P}_{\dot{\xi}/\alpha_0}"$ (we can assume $\dot{q}_0=\check{\emptyset}$).
    
        \item $\dom(f_1)$ is an $\alpha_1$-domain, and for $\gamma \in \dom(f_1)$, $f_1(\gamma)$ is 
 a $P_{\alpha_0}*\dot{P}_{\dot{\xi}/\alpha_0}$-name, $\Vdash_{P_{\alpha_0}*\dot{P}_{\xi}/\alpha_0} ``f_1(\gamma)<\alpha_1"$.
 \item For such a condition $p$, define $p \restriction P_{\alpha_0}=\langle f_0 \rangle$.
    \end{itemize}
    From now, we replace $\dot{\xi}$ by $\alpha_0$. We view $(\langle f_0 \rangle, \langle \dot{P}_{\alpha_0/\alpha_0},\dot{q}_0 \rangle)$ or $(\langle f_0 \rangle, \dot{q}_0)$ as a condition in $P_{\alpha_0}* \dot{P}_{\alpha_0/\alpha_0}$.
    We say that $$(\langle f_0 \rangle,\langle \dot{P}_{\alpha_0/\alpha_0},\dot{q}_0 \rangle) {}^\frown \langle f_1 \rangle\leq_{\alpha_1} (\langle g_0 \rangle,\langle \dot{P}_{\alpha_0/\alpha_0},\dot{r}_0 \rangle) {}^\frown \langle g_1 \rangle\text{ iff }$$ 
    $$(\langle f_0 \rangle,\langle \dot{P}_{\alpha_0/\alpha_0},\dot{q}_0\rangle) {}^\frown \langle f_1 \rangle\leq_{\alpha_1}^* (\langle g_0 \rangle,\langle \dot{P}_{\alpha_0/\alpha_0},\dot{r}_0\rangle) {}^\frown \langle g_1 \rangle\text{ iff }$$
    $f_0 \supseteq g_0, \dom(f_1) \supseteq \dom(g_1), \text{ and for } \gamma \in \dom(g_1), (\langle f_0 \rangle,\dot{q}_0) \Vdash_{P_{\alpha_0}*\dot{P}_{\alpha_0/\alpha_0}}``f_1(\gamma)=g_1(\gamma)".$
\end{enumerate}
\end{definition}

Let $\dot{C}_{\alpha_1}$ be a $P_{\alpha_1}$-name such that for $p$ of the first kind, $p \Vdash_{\alpha_1} \dot{C}_{\alpha_1}=\{\alpha_1\}$, and for $p$ of the second kind, $p \Vdash_{\alpha_1} ``\dot{C}_{\alpha_1}=\{\alpha_0,\alpha_1\}"$.
We now define different types of quotients.

\begin{itemize}
    \item $\dot{P}_{\alpha_1/\alpha_1}$ is a $P_{\alpha_1}$-name of the trivial forcing, with the obvious extension and the obvious direct extension. In $V^{P_{\alpha_1}}$, let $\dot{C}_{\alpha_1/\alpha_1}$ be a $\dot{P}_{\alpha_1/\alpha_1}$-name of the empty set.

\item The quotient $\dot{P}_{\alpha_1/\alpha_0}$ is a $P_{\alpha_0}$-name of the following forcing notion. Let $G$ be $P_{\alpha_0}$-generic. The forcing $P_{\alpha_1}[G]:=\dot{P}_{\alpha_1/\alpha_0}[G]$ consists of $(\langle P_{\alpha_0}[G], \emptyset \rangle) {}^\frown \langle f \rangle$ where $\Vdash_{\dot{P}_{\alpha_0/\alpha_0}[G]}``f \in C(\alpha_1^+,\alpha_1^{++})"$ ($C(\alpha_1^+,\alpha_1^{++})$ is considered in $(V[G])^{\dot{P}_{\alpha_0/\alpha_0}[G]}=V[G])$, and $\dom(f) \in V$ Note that $\emptyset$ is considered as the condition in $P_{\alpha_0}[G]$. The extension and the direct extension are the same and are defined as the following. We assume that for each $P_{\alpha_0}$ of a condition in $\dot{P}_{\alpha_1/\alpha_0}$ is of the form $p_0=(\langle \dot{P}_{\alpha_0/\alpha_0},\check{\emptyset}) {}^\frown \langle \dot{f} \rangle$. We say that $p \in P_{\alpha_0}$ {\em interprets $p_0$} if $p$ decides $\dot{\dom(f)}$. The collection of such $p$ is open dense and if $p$ interprets $p_0$, we may write $f$ where $\dom(f)$ is the domain where $p_0$ interprets $\dot(\dom{f})$. Then we can write $p_0 {}^\frown p_1$ as $(p_0,\langle \dot{P}_{\alpha_0/\alpha_0},\check{\emptyset}) {}^\frown \langle f \rangle$. For $p_0,p_1$ which are $P_{\alpha_0}$-name of conditions in $\dot{P}_{\alpha_1/\alpha_0}$, $\Vdash_{\alpha_0} ``p_0 \leq p_1$ iff $\exists p \in \dot{G}_{P_{\alpha_0}}$ $p$ interprets $p_0$ and $p_1$, and $p{}^\frown p_0 \leq p^\frown p_1"$. Back to the ground model, in $V^{P_{\alpha_0}}$, let $\dot{C}_{\alpha_1/\alpha_0}$ be the $\dot{P}_{\alpha_1/\alpha_0}$-name for $\{\alpha_1\}$. The point of having an empty set in the condition because it is more natural to translate a condition in $P_{\alpha_1}$ of the second kind to a condition in $\dot{P}_{\alpha_1/\alpha_0}$, namely, for each $p=(\langle f_0 \rangle,\langle \dot{P}_{\dot{\xi}/\alpha_0},\dot{q}_0\rangle) {}^\frown \langle f_1 \rangle$ in $P_{\alpha_1}$, we have that $\Vdash_{\alpha_0}``(\langle \dot{P}_{\alpha_0/\alpha_0},\dot{q} \rangle){}^\frown \langle f_1 \rangle \in \dot{P}_{\alpha_1/\alpha_0}"$. This is because $\dot{q}$ is always interpreted as the empty set in $\dot{P}_{\alpha_0/\alpha_0}$, and $f_1$ is a function whose range contains names of ordinals in with respect to the correct forcing. Note that $\{p \in P_{\alpha_1} \mid p \text{ is of the second kind}\}$ can be densely embedding in $P_{\alpha_0}* \dot{P}_{\alpha_1/\alpha_0}$ in the sense of $\leq$.

\end{itemize}

  The subforcing of $P_{\alpha_1}$ containing conditions of second kinds is nothing but a two-step iteration of the Cohen forcings, except that the domains can always be decided by the weakest element to be in the ground model.

  \begin{definition}
      The {\em forcings at level $\alpha$} are the forcings of the form $P_\alpha$ or $P_{\alpha/\beta}$.
  \end{definition}

\section{The induction scheme}
\label{theindsch}
We are now stating the induction scheme, and point out that it holds for the basic cases.

\begin{proposition}[The induction scheme]
\label{indsch}
Let $\alpha$ be an inaccessible cardinal. Here are the properties for the forcings at level $\alpha$.

\begin{enumerate}
    \item \label{1} The basic properties of the forcing $(P_\alpha,\leq, \leq^*)$.
    \begin{itemize}
        \item $|P_\alpha| = \alpha^{++}$.
        \item $(P_\alpha,\leq)$ is $\alpha^{++}$-c.c.
        \item $(P_\alpha,\leq,\leq^*)$ has the Prikry property.
    \end{itemize}
    \item \label{2} The $P_\alpha$-name of the set $\dot{C}_\alpha$. Let $C_\alpha=\dot{C}_\alpha[G]$ where $G$ is generic over $P_\alpha$.
    \begin{itemize}
        \item $C_\alpha \subseteq \alpha+1$, $\max(C_\alpha)=\alpha$.
        \item If $\circ(\alpha)=0$, then $C_\alpha \cap \alpha$ is a bounded subset of $\alpha$.
        \item If $\circ(\alpha)>0$, then $C_\alpha \cap \alpha$ is a club subset of $\alpha$.
        \item $C_\alpha$ contains only inaccessible cardinals of $V$.
    \end{itemize}
    \item  \label{3} Cardinals and cofinalities in the extension.
    \begin{itemize}
        \item If $\circ(\alpha)=0$, then $\alpha$ remains regular in the extension over $P_\alpha$.
        \item If $\circ(\alpha)>0$, then when we force over $P_\alpha$, $\alpha$ is singularized and $\cf(\alpha)=\cf(\omega^{\circ(\alpha)})$ (the ordinal exponentiation).
        \item In the extension, for every cardinal $\beta \leq \alpha$, $2^\beta=\beta^+$ or $2^\beta=\beta^{++}$, and $2^\beta=\beta^{++}$ iff $\beta \in \lim(C_\alpha)$.
        \item For each $V$-regular $\beta \leq \alpha$, $\beta$ is singularized iff $\beta \in \lim(C_\alpha)$.
    \end{itemize}
    \item \label{4} $\dot{P}_{\alpha/\alpha}$ is always a $P_\alpha$-name of the trivial forcing $(\{\emptyset\},\leq,\leq^*)$.
    \item \label{5} The factor $\dot{P}_{\alpha/\beta}$ for $\beta< \alpha$.
    \begin{itemize}
        \item $\{p \in P_\alpha \mid p \restriction P_\beta$ exists$\}$ densely embeds into $P_\beta*\dot{P}_{\alpha/\beta}$ in the $\leq$ sense.
        \item $\Vdash_\beta `` |\dot{P}_{\alpha/\beta}| =\alpha^{++}, (\dot{P}_{\alpha/\beta},\leq)$ is $\alpha^{++}$-c.c.".
        \item $\Vdash_\beta ``(\dot{P}_{\alpha/\beta},\leq^*)$ is $\beta^*$-closed", where $\beta^*=\min\{\xi>\beta \mid \xi$ is strongly inaccessible$\}$.
        \item $\Vdash_\beta ``(\dot{P}_{\alpha/\beta},\leq,\leq^*)$ has the Prikry property".
    \end{itemize}
    \item \label{6} The quotient set $C_{\alpha/\beta}$: Let $G$ be $P_\beta$-generic over $V$ and $H$ be $\dot{P}_{\alpha/\beta}[G]$-generic over $V[G]$. Let $C_{\alpha/\beta}=\dot{C}_{\alpha/\beta}[G][H]$. 
    \begin{itemize}
        \item If $\beta=\alpha$, then $C_{\alpha/\beta}=\emptyset$. 
        \item Suppose $\beta<\alpha$. Then $I=G*H$ is $P_\alpha$-generic, which introduces the set $C_\alpha$. Also, $G$ introduces the set $C_\beta$.
        Then $C_{\alpha/\beta} \subseteq (\beta,\alpha]$, and $C_\alpha=C_\beta \sqcup C_{\alpha/\beta}$.
    \end{itemize}
   \item  \label{7} Double quotients: Let $\gamma \leq \beta \leq \alpha$ and $G$ is $P_\gamma$-generic. Then $\dot{P}_{\alpha/\beta}[G]$ is defined as 
   $$\Vdash_{P_\beta[G]}``p \in \dot{P}_{\alpha/\beta}[G] \text{ iff } p \in P_\alpha[G*\dot{H}]",$$
   where $\dot{H}$ is the canonical $P_{\beta}[G]$-generic.
 \end{enumerate} 
 \end{proposition}

 We always skip (\ref{4}) and (\ref{7}) of Proposition \ref{indsch} since they will follow directly from the definitions.
 Showing the induction scheme of the forcings at level the first inaccessible cardinal is easy. For a non-triviality, we now show that the forcing $P_{\alpha_1}$ as described in Definition \ref{twolvl} satisfies the induction scheme.

 \begin{proposition}
 \label{lis}
     Let $\alpha_0<\alpha_1$ be the first two inaccessible cardinals. Then forcings at level $\alpha_1$ satisfies the induction scheme.
 \end{proposition}

 \begin{proof}
\begin{enumerate}
    \item \begin{itemize}
    \item The set of conditions in $P_{\alpha_1}$ of the first kind is $C(\alpha_1^+,\alpha_1^{++})$, whose size is $\alpha^{++}$. Conditions of the second kind are of the form $(\langle f_0 \rangle, \langle \dot{P}_{\dot{\xi}/\alpha_0},\dot{q}_0 \rangle) {}^\frown \langle f_1 \rangle$. We assume that the names are in their simplest form in the sense that $\dot{\xi}=\check{\alpha_0}$, $\dot{q}_0=\check{\emptyset}$. The part $(\langle f_0 \rangle, \langle \dot{P}_{\dot{\xi}/\alpha_0},\dot{q}_0 \rangle)$ is in $V_{\alpha_1}$. Then for each $\gamma \in \dom(f_1)$, $f_1(\gamma)$ is a $P_{\alpha_0}*\dot{P}_{\alpha_0/\alpha_0}$-name of an ordinal below $\alpha$. By replacing $f_1(\gamma)$ with its nice name, assume that $f_1(\gamma) \in V_{\alpha_1}$. Hence, the number of such $f_1$'s is $(\alpha_1^{++})^{\alpha_1}=\alpha_1^{++}$. Hence, $|P_{\alpha_1}|=\alpha_1^{++}$.

    \item Suppose that $X=\{ p^\gamma \mid \gamma<\alpha_1^{++} \}$ is an antichain of conditions in $P_{\alpha_1}$. By shrinking $X$, we may assume that $X$ contains conditions of the same kind. If it contains conditions of the first kind, then the standard $\Delta$-system applies. Suppose $X$ contains conditions of the second kind. By shrinking further, assume there is $p_0$ such that for every $\gamma$, $p^\gamma=p_0 {}^\frown \langle f_1^\gamma \rangle$. Then we can apply a standard $\Delta$-system argument on $\{ f_1^\gamma \mid \gamma<\alpha_1^{++} \}$, and we are done.

    \item Obvious, since $\leq$ and $\leq^*$ on $P_{\alpha_1}$ are the same.
    
    \end{itemize}

 \item Note that $\circ(\alpha_1)=0$. If $G$ contains conditions of the first kind, then $C_{\alpha_1}=\{\alpha_1\}$, and if $G$ contains conditions of the second kind, then $C_{\alpha_1}=\{\alpha_0,\alpha_1\}$. In either case, it is a subset of $\alpha_1+1$ whose maximum is $\alpha_1$. Also, $C_{\alpha_1} \cap \alpha_1$ is either $\emptyset$ or $\{\alpha_0\}$ which is bounded in $\alpha_1$, and $C_{\alpha_1}$ contains only inaccessible cardinals in $V$.

\item $\circ(\alpha_1)=0$, and the forcing $P_{\alpha_1}$ is equivalent to either a Cohen forcing $\Add(\alpha_1^+,\alpha_1^{++})$, or a two-step iteration of Cohen forcings $\Add(\alpha_0^+,\alpha_0^{++})* \Add(\alpha_1^+,\alpha_1^{++})$. In both cases, $\alpha_1$ remains regular, GCH still holds, and $\lim(C_\alpha)=\{\emptyset\}$.

\item $\dot{P}_{\alpha_1/\alpha_1}$ is a $P_{\alpha_1}$-name of the trivial forcing.

\item Consider $\dot{P}_{\alpha_1/\alpha_0}$.
\begin{itemize}
    \item For each $p=(\langle f_0 \rangle,\langle \dot{P}_{\dot{\xi}/\alpha_0},\dot{q}_0\rangle) {}^\frown \langle f_1 \rangle$, consider the map $\pi(p)=(\langle f_0 \rangle,(\langle \dot{q} \rangle) {}^\frown f_1 \rangle)$. Clearly, this map is a dense embedding from $\{p \in P_{\alpha_1} \mid p \restriction P_{\alpha_0}\}$ to $P_{\alpha_0}* \dot{P}_{\alpha_1/\alpha_0}$.
    \item Since $P_{\alpha_0}$ forces GCH, a similar argument as in $(1)$ shows that $\Vdash_{\alpha_0} `` |\dot{P}_{\alpha_1/\alpha_0}|=\alpha_1^{++}$, $(P_{\alpha_1/\alpha_0},\leq)$ is $\alpha_1^{++}$-c.c., "
    \item Let $G$ be $P_{\alpha_0}$-generic. Conditions in $P_{\alpha_1}[G]$ are of the form $(\langle \emptyset \rangle) {}^\frown \langle f_1 \rangle$. We ignore the empty set's part. Note that since $P_{\alpha_0}[G]:=\dot{P}_{\alpha_0/\alpha_0}[G]$ is trivial, so $f_1$ is just a Cohen condition in $V[G]$. We now assume that a condition in $P_{\alpha_1}[G]$ is $\langle f_1 \rangle$. Let $\langle f_1^\gamma \mid \gamma<\gamma^*\rangle$ be a decreasing sequence of conditions, where $\gamma^*<\alpha_1$. In $V$, let $d^*=\cup_{\gamma<\gamma^*} \{d \mid \exists p \in P_{\alpha_0}(p$ decides $\dom(f_1^\gamma)$ as $d)\}$. Then $d^* \in V$, and let $f^*$ be such that $\dom(f^*)=d^*$, and in $V[G]$, $f^* \leq f_1^\gamma$ for all $\gamma$. Then $f^*$ is as required.
    \item $\Vdash_{\alpha_0} ``\leq, \leq^*$ are the same in $\dot{P}_{\alpha_1/\alpha_0}$, hence has the Prikry property".
    \end{itemize}
    \item In $V^{P_{\alpha_1}*\dot{P}_{\alpha_1/\alpha_1}}$, $C_{\alpha_1/\alpha_1}$ is the empty set. In $V^{P_{\alpha_0}*\dot{P}_{\alpha_1/\alpha_0}}$, $C_{\alpha_1/\alpha_0}=\{\alpha_1\} \subseteq (\alpha_0,\alpha_1]$, and in this model, $C_{\alpha-0} \sqcup C_{\alpha_1/\alpha_0}=C_{\alpha_1}$, since it is the same model with the extension $V^{P_{\alpha_1}}$ using conditions of the second kind.
    \item Trivial since the definition is given.
\end{enumerate} 
 \end{proof}

 \begin{remark}
    \begin{enumerate}
        \item $P_{\alpha_0}* \dot{P}_{\alpha_1/\alpha_0}$ is equivalent to the subforcing $P_{\alpha_1}$ containing conditions of the second kind, and there is a natural translation from one generic to another. Namely, suppose that $G*H$ is such a generic object. Define $I=\{(p_0,\langle \dot{P}_{\alpha_0/\alpha_0},\dot{q} \rangle){}^\frown p_1 \mid p_0 \in G, \Vdash_{\alpha_0} ``(\langle \dot{P}_{\alpha_0/\alpha_0},\dot{q} \rangle) {}^\frown p_1 \in \dot{H}"\}$. Then $V[I]=V[G*H]$.
        \item If we force with conditions in $P_{\alpha_1}$ of the second kind, we can obtain an equivalent generic object from $P_{\alpha_0}* \dot{P}_{\alpha_1/\alpha_0}$ naturally. Namely, if $I$ is $P_{\alpha_1}$-generic containing conditions of the second kind, let $$G=\{\langle f_0 \rangle \mid \exists \dot{q},f_1 (\langle f_0,\langle \dot{P}_{\alpha_0/\alpha_0},\dot{q}\rangle){}^\frown \langle f_1 \rangle \in I\},$$ 
        and 
        $$H=\{(\langle \emptyset \rangle)^\frown \langle f_1[G] \rangle \mid \exists f_0,\dot{q} (\langle f_0,\langle \dot{P}_{\alpha_0/\alpha_0},\dot{q}\rangle){}^\frown \langle f_1 \rangle \in I\}.$$
        Then $G$ is $P_{\alpha_0}$-generic, $H$ is $P_{\alpha_1}[G]$-generic, and $V[I]=V[G*H]$.
        \end{enumerate}
\end{remark}

\section{Below the first measurable cardinal}
\label{belthefir}
Let $\alpha$ be a strongly inaccessible cardinal which is below the first $\alpha^*$ with $\circ(\alpha^*)=1$. We will assume that $\alpha$ is at least the $\omega+1$-th strongly inaccessible cardinal so that the conditions of arbitrarily length will appear at this stage.

\begin{definition}
    $P_\alpha$ consists of the conditions of the following kinds:

    \begin{itemize}
        \item The {\em pure conditions}, which are conditions of the form $\langle f \rangle$, where $f \in C(\alpha^+,\alpha^{++})$.

        \item The {\em impure conditions}, which are conditions of the form

        $$(\langle f_0 \rangle {}^\frown \langle \dot{P}_{\dot{\beta_0}/\alpha_0},\dot{q}_0\rangle ) {}^\frown \cdots {}^\frown (\langle f_{n-1} \rangle {}^\frown \langle \dot{P}_{\dot{\beta}_{n-1}/\alpha_{n-1}},\dot{q}_{n-1}\rangle){}^\frown \langle f \rangle,$$
        for some $n>0$, where
        \begin{itemize}
            \item $\alpha_0< \cdots <\alpha_{n-1}<\alpha$ are inaccessible.
            \item for all $i$, $\Vdash_{\alpha_i} ``\alpha_i \leq \dot{\beta_i}<\alpha_{i+1}"$, where $\alpha_n=\alpha$.
            \item $f_0 \in C(\alpha_0^+,\alpha_0^{++})$ and for $i>0$, $\dom(f_i)=d_i$ is an $\alpha_i$-domain (in the sense of $V$), and for $\zeta \in d_i$, $f_i(\zeta)$ is a $P_{\alpha_{i-1}}*\dot{P}_{\dot{\beta}_{i-1}/\alpha_{i-1}}$-name and $$\Vdash_{P_{\alpha_{i-1}}* \dot{P}_{\dot{\beta}_{i-1}/\alpha_{i-1}}}``f_i(\zeta)<\alpha_i".$$ In particular, $$\Vdash_{P_{\alpha_{i-1}}* \dot{P}_{\dot{\beta}_{i-1}/\alpha_{i-1}}}``f_i \in \dot{C}(\alpha_i^+,\alpha_i^{++}).$$
            \item $\dom(f)=d$ is an $\alpha$-domain, and for $\zeta \in d$, $f(\zeta)$ is a $P_{\alpha_{n-1}}*\dot{P}_{\dot{\beta}_{n-1}/\alpha_{n-1}}$-name and $$\Vdash_{P_{\alpha_{n-1}}*\dot{P}_{\dot{\beta}_{n-1}/\alpha_{n-1}}}``f(\zeta)<\alpha".$$
            In particular, 
            $$\Vdash_{P_{\alpha_{n-1}}*\dot{P}_{\dot{\beta}_{n-1}/\alpha_{n-1}}} ``f \in \dot{C}(\alpha^+,\alpha^{++})".$$
            \item for all $i$, $\Vdash_{\alpha_i} ``\dot{q}_i \in \dot{P}_{\dot{\beta}_i/\alpha_i}"$.
        \end{itemize}
    \end{itemize}
    By recursion, we consider 
    $$(\langle f_0 \rangle {}^\frown \langle \dot{P}_{\dot{\beta_0}/\alpha_0},\dot{q}_0\rangle ) {}^\frown \cdots {}^\frown \langle f_{i} \rangle$$
    as a condition in $P_{\alpha_i}$. Denote $p \restriction P_{\alpha_i}$ as the condition as bove. We also consider 
    $$(\langle f_0 \rangle {}^\frown \langle \dot{P}_{\dot{\beta_0}/\alpha_0},\dot{q}_0\rangle ) {}^\frown \cdots {}^\frown (\langle f_{i} \rangle,\langle \dot{P}_{\dot{\beta}_i/\alpha_i},\dot{q}_i\rangle)$$
    as a condition in $P_{\alpha_i} * \dot{P}_{\dot{\beta}_i/\alpha_i}$. Denote such a condition by $p \restriction (i+1)$.
\end{definition}

The ordering $\leq_\alpha$ and $\leq_{\alpha}^*$ will be the same. We only define $\leq_\alpha$. When we mention a condition $p$, we put the superscript $p$ to every component in the condition. If $p$ is the condition as in the definition, we write $n^p=n$, $\tp(p)=\langle f \rangle$.

\begin{definition}
    Let $$p_0=(\langle f_0 \rangle {}^\frown \langle \dot{P}_{\dot{\beta_0}/\alpha_0},\dot{q}_0\rangle ) {}^\frown \cdots {}^\frown (\langle f_{n-1} \rangle {}^\frown \langle \dot{P}_{\dot{\beta}_{n-1}/\alpha_{n-1}},\dot{q}_{n-1}\rangle){}^\frown \langle f \rangle,$$
    and
    $$p_1=(\langle g_0 \rangle {}^\frown \langle \dot{P}_{\dot{\xi_0}/\gamma_0},\dot{r}_0\rangle ) {}^\frown \cdots {}^\frown (\langle g_{n-1} \rangle {}^\frown \langle \dot{P}_{\dot{\xi}_{n-1}/\gamma_{n-1}},\dot{r}_{m-1}\rangle){}^\frown \langle g \rangle.$$
    We say that $p_0 \leq_\alpha p_1$ iff

    \begin{itemize}
        \item $n=m$.
        \item for $i<n$, $\alpha_i=\gamma_i$.
        \item $f_0 \supseteq g_0$, $\langle f_0 \rangle \Vdash_{\alpha_0} ``\dot{\beta}_0=\dot{\xi}_0 \text{ and } \dot{q}_0 \leq_{\dot{\beta}_0/\alpha_0} \dot{r}_0"$ (we can assume $\dot{\beta}_0=\dot{\xi}_0$).
        \item for $i>0$, $d_i^{p^0} \supseteq d_i^{p^1}$, and for $\zeta \in d_i^{p^1}$, $p \restriction i \Vdash_{P_{\alpha_i}*\dot{P}_{\dot{\beta}_i/\alpha_i}} ``f_i(\zeta)=g_i(\zeta)".$
        \item for $i>0$, $(p_0 \restriction i){}^\frown \langle f_i \rangle \Vdash_{\alpha_i} ``\dot{\beta}_i=\dot{\xi}_i$ and $\dot{q}_i \leq_{\dot{\beta}_i/\alpha_i} \dot{r}_i"$ (we can assume $\dot{\beta}_i=\dot{\xi}_i$).
        \item $\dom(f) \supseteq \dom(g)$ and for $\zeta \in \dom(g)$, $$p_0 \restriction n \Vdash_{P_{\alpha_{n-1}}*\dot{P}_{\dot{\beta}_{n-1}/\alpha_{n-1}}}``f(\zeta)=g(\zeta)".$$
    \end{itemize}

\end{definition}

We may also assume that $\dot{\xi}_i=\dot{\beta}_i$ for all $i$. The extension relation does not increase the length of a condition. For a generic $G$ containing a condition $p$, define $C_\alpha$ as the following: If $p$ is pure, then $C_\alpha=\{\alpha\}$. Assume $p$ is impure and $n=n^p$. Then $p \restriction n \in P_{\alpha_n}*\dot{P}_{\dot{\beta}_n/\alpha_n}$. Let $\beta_n=\dot{\beta}_n[G \restriction P_{\alpha_{n-1}}]$. By Proposition \ref{indsch} (\ref{2}) and (\ref{6}), $G \restriction (P_{\alpha_n}*\dot{P}_{\dot{\beta}_n/\alpha_n})$ introduces the set $C^\prime=C_{\alpha_{n-1}} \sqcup C_{\beta_{n-1}/\alpha_{n-1}} \subseteq \beta_{n-1}+1$ with $\max(C^\prime)=\beta_{n-1}$. Define $C_\alpha=C^\prime \cup \{\alpha\}$.
Still, this forcing does not change the cardinal arithmetic.
    
We now define $P_{\alpha/\beta}$ for $\beta \leq \alpha$. A key point is that we need $\{p \in P_\alpha \mid p \restriction P_\beta$ is defined$\}$ to be densely embedded in $P_\beta*\dot{P}_{\alpha/\beta}$.

\begin{definition}[The quotient forcing]
Let $\dot{P}_{\alpha/\alpha}$ be the $P_\alpha$-name of the trivial forcing $(\{\emptyset\},\leq, \leq^*)$. In $V^{P_{\alpha}}$, let $\dot{C}_{\alpha/\alpha}$ be the $\dot{P}_{\alpha/\alpha}$-name of the empty set.
Now assume that $\beta<\alpha$.
Define $\dot{P}_{\alpha/\beta}$ as the following.
Let $G$ be $P_\beta$-generic. Define $P_\alpha[G]=\dot{P}_{\alpha/\beta}[G]$ as the forcing consisting of conditions of the form

$p=(\langle P_{\beta^\prime}[G],q^\prime\rangle) {}^\frown (\langle f_0 \rangle, \langle \dot{P}_{\dot{\beta}_0/\alpha_0}[G],\dot{q}_0\rangle)\cdots {}^\frown (\langle f_{n-1} \rangle, \langle \dot{P}_{\dot{\beta}_{n-1}/\alpha_{n-1}}[G],\dot{q}_{n-1}\rangle) {}^\frown \langle f \rangle$ where $n\geq 0$ and

\begin{enumerate}
    \item $\beta \leq \beta^\prime<\alpha$, so $P_{\beta^\prime}[G]$ was already defined by recursion, which is just $\dot{P}_{\dot{\beta^\prime}[G]/\beta}[G]$ and $\beta^\prime=\dot{\beta}^\prime[G]$. Furthermore, $q^\prime \in P_{\beta^\prime}[G]$.
    \item If $n>0$, then $\alpha_0<\cdots<\alpha_{n-1}$, and for $i<n$, 
    \begin{itemize}
        \item let $d_i=\dom(f_i)$, then $d_i$ is an $\alpha_i$-domain, $d_i \in V$.
        \item for $\zeta \in d_0$, $\Vdash_{P_{\beta^\prime}[G]}``f_0(\zeta)<\alpha_0"$, and if $i>0$, then for $\zeta \in d_i$, $\Vdash_{P_{\alpha_{i-1}}[G]*\dot{P}_{\dot{\beta}_{i-1}/\alpha_{i-1}}[G]}``f_i(\zeta)<\alpha_i"$.
        \item $\Vdash_{P_{\alpha_i}[G]}``\alpha_i\leq \dot{\beta}_i<\alpha_{i+1}"$, where $\alpha_n=\alpha$.
        \item $\Vdash_{P_{\alpha_i}[G]}``\dot{q}_i \in \dot{P}_{\dot{\beta}_i/\alpha_i}[G]"$.
    \end{itemize}
    \item $d:=\dom(f)$ is an $\alpha$-domain, and is in $V$.
    \item Fix $\zeta \in d$. If $n=0$, then $\Vdash_{P_{\beta^\prime}[G]}``f(\zeta)<\alpha"$, otherwise, $\Vdash_{P_{\alpha_{n-1}}[G]*P_{\dot{\beta}_{n-1}/\alpha_{n-1}}[G]} ``f(\zeta)<\alpha"$.
\end{enumerate}
Back in $V$. If $\dot{p}$ is a $P_\beta$-name of a condition in $ \dot{P}_{\alpha/\beta}$, then by density, there is $p_0 \in P_\beta$ such that $p_0$ decides $n$, $\alpha_0, \cdots, \alpha_{n-1}$, $\dom(f_0), \cdots, \dom(f_{n-1})$, $\dom(f)$. In this case, we say that $p_0$ {\em interprets} $\dot{p}$.
All in all, for such $p_0$ which interprets all the relevant components of $\dot{p}$, let $p_1$ be such the interpretation.
Write $p_0$ as $r_0 {}^\frown \langle g \rangle$ and by the interpretation, we write $$p_1=(\langle \dot{P}_{\beta^\prime/\beta},\dot{q}^\prime \rangle) {}^\frown (\langle f_0 \rangle, \langle \dot{P}_{\dot{\beta}_0/\alpha_0},\dot{q}_0\rangle)\cdots {}^\frown (\langle f_{n-1} \rangle, \langle \dot{P}_{\dot{\beta}_{n-1}/\alpha_{n-1}},\dot{q}_{n-1}\rangle) {}^\frown \langle f \rangle.$$
There is a natural concatenation $p_0$ with $p_1$, written by $p_0{}^\frown p_1$, which is 
$$r=r_0 {}^\frown (\langle g \rangle, \langle \dot{P}_{\beta^\prime/\beta},\dot{q}^\prime \rangle) {}^\frown \cdots {}^\frown (\langle f_{n-1} \rangle, \langle \dot{P}_{\dot{\beta}_{n-1}/\alpha_{n-1}},\dot{q}_{n-1}\rangle) {}^\frown \langle f \rangle.$$
Then $r \in P_\alpha$ with $r \restriction P_\beta=p_0$ exists.
For $p_0$ and $p_1$ in $\dot{P}_{\alpha/\beta}$, we say that $p_0 \leq p_1$ if there is $p \in G^{P_\beta}$ such that $p$ interprets $p_0$ and $p_1$, and $p {}^\frown p_0 \leq_\alpha p {}^\frown p_1$. Also define $p_0 \leq^* p_1$ if there is $p \in G^{P_\beta}$ such that $p$ interprets $p_0$ and $p_1$, and $p {}^\frown p_0 \leq^*_\alpha p{}^\frown p_1$ (note that at this level $\leq^*$ and $\leq$ are still the same). One can check that the map $\phi: \{p \in P_\alpha \mid p \restriction P_\beta$ exists$\} \to P_\beta* \dot{P}_{\alpha/\beta}$ defined by
$\phi(p)=(p \restriction P_\beta, p \setminus P_\beta)$
is a dense embedding, where $p \setminus P_\beta$ is the obvious component of $p$ which is in $\dot{P}_{\alpha/\beta}$.

In $V^{P_\beta}$, let $\dot{C}_{\alpha/\beta}$ be a $\dot{P}_{\dot{\beta}/\alpha}$-name of the set described as the following. Let $G$ be $P_\beta$-generic.
Write 
$$p=(\langle P_{\beta^\prime}[G],q^\prime) {}^\frown (\langle f_0 \rangle, \langle \dot{P}_{\dot{\beta}_0/\alpha_0}[G],\dot{q}_0\rangle)\cdots {}^\frown (\langle f_{n-1} \rangle, \langle \dot{P}_{\dot{\beta}_{n-1}/\alpha_{n-1}}[G],\dot{q}_{n-1}\rangle) {}^\frown \langle f \rangle$$
as an element in $P_\alpha[G]$.
The part which excludes the top part, i.e.
$$(\langle P_{\beta^\prime}[G],q^\prime) {}^\frown (\langle f_0 \rangle, \langle \dot{P}_{\dot{\beta}_0/\alpha_0}[G],\dot{q}_0\rangle)\cdots {}^\frown (\langle f_{n-1} \rangle, \langle \dot{P}_{\dot{\beta}_{n-1}/\alpha_{n-1}}[G],\dot{q}_{n-1}\rangle)$$
is in $P_{\alpha_{n-1}}[G]*\dot{P}_{\dot{\beta}_{n-1}/\alpha_{n-1}}[G]$. Let $H$ be generic over the forcing.
By our induction scheme, $H$ produces $C_0 \sqcup C_1$, where $C_0 \subseteq (\beta,\alpha_{n-1}]$ (can be empty if $n=0$), and $C_1 \subseteq (\alpha_{n-1},\beta_{n-1}]$ (can be empty if $\beta_{n-1}$, the interpretation of $\dot{\beta}_{n-1}$, is $\alpha_{n-1}$).
If $n>0$, then $\max(C_0)=\alpha_{n-1}$, and if $\beta_{n-1}>\alpha_{n-1}$, then $\max(C_1)=\beta_{n-1}$.
Let $C_{\alpha/\beta}=C_0 \cup C_1 \cup \{\alpha\}$.

\end{definition}

\begin{proposition}
$P_\alpha$ and the relevant quotients at $\alpha$ satisfy Proposition \ref{indsch}.
\end{proposition}

\begin{proof}
\begin{enumerate}
    \item Similar as the proof of the corresponding properties in Propoisition \ref{lis}.
    \item $\circ(\alpha)=0$. Then the forcing $P_\alpha$ introduces the set $C_\alpha \subseteq \alpha+1$ where $C_\alpha \setminus \{\alpha\}$ is a bounded subset of $\alpha$. By induction hypothesis, it is easy to see that $C_\alpha$ contains only inaccessible cardinals in $V$.
    
    \item The forcing $P_\alpha$ under a certain condition can be factored to $Q*\dot{C}(\alpha^+,\alpha^{++})$, where $Q \in V_\alpha$, and hence, $\alpha$ is still regular. Note that since $\alpha$ is below the first measurable cardinal, we can still induct to show that $C_\alpha$ is finite. Since $Q$ is either empty or a two-step iteration where it forces GCH. Hence, $P_\alpha$ still forces GCH.

    \item Obvious.

    \item Let $\beta<\alpha$.
    \begin{itemize}
    \item The map $p \mapsto (p \restriction P_\beta, p \setminus P_\beta)$ is a dense embedding from $\{p \in P_\alpha \mid p \restriction P_\beta$ exists$\}$ to $P_\beta*\dot{P}_{\alpha/\beta}$.

    \item Similar to the proof of the corresponding properties in Proposition \ref{indsch}, $\Vdash_\beta ``|\dot{P}_{\alpha/\beta}|=\alpha^{++}$ and is $\alpha^{++}$-c.c.$"$

    \item Let $\beta^\prime<\beta^*$ and $\Vdash_\beta ``\{p^\gamma \mid \gamma<\beta^\prime\}$ be a $\leq^*$-decreasing sequence of conditions in $\dot{P}_{\alpha/\beta}"$. We may assume that $p^\gamma=p_0^\gamma {}^\frown \langle f^\gamma \rangle$.
    Then $\Vdash_\beta ``\{p_0^\gamma \mid \gamma<\beta^\prime\}$ is a $\leq^*$-decreasing sequence in a certain forcing $P_{\alpha^*}*\dot{P}_{\dot{\beta^*}/\alpha^*}"$. By induction hypothesis, the two-step iteration is $\beta^*$-closed under $\leq^*$. Let $p_0^*$ be such that for all $\gamma$, $\Vdash_\beta ``p_0^* \leq^* p_0^\gamma"$. Now a similar proof as in the corresponding property of Proposition \ref{indsch} can be used to find $f_1^*$ such that for all $\gamma$, $\Vdash_\beta ``p_0^* {}^\frown \langle f_1^* \rangle \leq^* p_0^\gamma {}^\frown \langle f_1^\gamma \rangle"$.

    \item Since $\leq$ and $\leq^*$ on $\dot{P}_{\alpha/\beta}$ coincide, the Prikry property holds.
    \end{itemize}
    \item By the construction of $\dot{C}_{\alpha/\beta}$ and the factorization, the property holds.
    \item Obvious by the definition of the double quotient stated in the Proposition \ref{indsch}.
\end{enumerate}    
\end{proof}

\section{At the first $\alpha$ with $\circ(\alpha)=1$}
\label{atfir}
We exhibit the forcing at the level of the first cardinal with a positive Mitchell order. Let $\alpha$ be the first such that $\circ(\alpha)=1$. A variation of the Extender-based Prikry forcing will be introduced. Instead of diving into a full definition all at once, we progress through a series of definitions. We make use of the digression we did in Section \ref{bas}.

\begin{definition}
    A {\em pure condition} of $P_\alpha$ is $p=\langle f_0, \vec{f}, A, F \rangle$ where there is a {\em common domain $d$} such that

\begin{enumerate}
    \item $A$ is a $d$-tree.
    \item $\dom(F)=A(\alpha)$.
    \item for $\nu \in \dom(F)$, $F(\nu)=\langle \dot{P}_{\dot{\beta}_\nu/\nu},\dot{q}\rangle$ where $\Vdash_\nu ``\nu \leq \dot{\beta}_\nu<\alpha \text{ and } \dot{q} \in \dot{P}_{\dot{\beta}_\nu/\nu}".$ We often represent $F(\nu)$ as $\langle F(\nu)_0,F(\nu)_1 \rangle$.
    \item $\dom(f)=d$ and $f_0 \in C(\alpha^+,\alpha^{++})$.
    \item $\vec{f}=\langle f_\nu \mid \nu \in A(\alpha) \rangle$.
    \item for each $\nu \in A(\alpha)$, $\dom(f_\nu)=d$ and for $\zeta \in d$, $f(_\nu(\zeta)$ is a $P_\nu*\dot{P}_{\dot{\beta}_\nu/\nu}$-name and $\Vdash_{P_\nu*\dot{P}_{\dot{\beta}_\nu/\nu}} ``f_\nu(\zeta)<\alpha"$. 
\end{enumerate}

\end{definition}

The forcing looks different from a usual Extender-based forcing. The main difference is that now we have a sequence of Cohen-like functions. The role of the sequence of the Cohen-like functions is that we want the closure of the quotient forcings at this level (and also in general) to be high with respect to the direct extension relation. If we just use a Cohen function in the ground model, then the corresponding quotient will no longer be highly closed with respects to the direct extension relation. When we perform a one-step extension, we want to somehow change the Cohen function to a name, where the name of a Cohen function  respects to the forcing in which the lower part lives. The explanation will make a bit more sense once we introduce the one-step extension operation.

We now discuss a one-step extension of a pure condition. Suppose that $p=\langle f_0,\vec{f},A,F \rangle$ with the common domain $d$.
Let $\langle \mu \rangle \in \Lev_0(A)$ with $\mu(\alpha)=\nu$. The {\em one-step extension of $p$ by $\mu$} is $r {}^\frown \langle g_0,\vec{g},A^\prime, F^\prime \rangle$ such that 
\begin{itemize}
    \item $r=(\langle f_0 \circ \mu^{-1} \rangle,F(\nu))$. Write $F(\nu)=\langle \dot{P}_{\dot{\beta}_\nu/\nu},\dot{q}\rangle$ (note that we assume $f_0 \circ \mu^{-1} \in C(\nu^+,\nu^{++})$ and the collection of such $\mu$ is large).
    \item $A^\prime=\{\vec{\tau} \in A_{\langle \mu \rangle} \mid \tau_0(\alpha)>\beta^*\}$ where $\beta^*=\sup\{\gamma \mid \exists r \in P_\nu(r\Vdash_\nu``\dot{\beta}_\nu=\gamma)"\}$.
    \item $F^\prime=F \restriction (A^\prime(\alpha))$.
    \item $\dom(g_0)=d$.
    \item $ \Vdash_{P_\nu*\dot{P}_{\dot{\beta}/\nu}} ``g_0=f_\nu \oplus \mu"$, i.e. for $\zeta \in d$, if $\zeta \in \dom(\mu)$, $g_0(\zeta)=\check{\mu(\alpha)}$, otherwise, $\Vdash_{P_\nu*\dot{P}_{\dot{\beta}/\nu}}``g_0(\zeta)=f_\nu(\zeta)"$ (we can assume tat $g_0(\zeta)=f_\nu(\zeta)$ for $\zeta \in d \setminus \dom(\mu))$.
    \item $\vec{g}=\langle f_{\nu^\prime} \mid \nu^\prime \in A^\prime(\alpha) \rangle$.

\end{itemize}

 Note that particular, $\langle f_0\circ \mu^{-1} \rangle \in P_\nu$, and so, $r$ can be considered as a condition in $P_\nu * \dot{P}_{\dot{\beta}_\nu/\nu}$. Like in a lot of Pirkry-type forcings, a $d$-tree at $\alpha$ gives us objects to create new blocks below $\alpha$. The part $\langle g_0,\vec{g},A^\prime,F^\prime \rangle$ looks similar to a pure condition except that for each $\zeta$, we now have that each $g_0(\zeta)$ is a name with respects to the forcing corresponding to where $r$ lives.

We now define a condition in a general form.

\begin{definition}

A condition in $P_\alpha$ is either pure or {\em impure}, which is of the form

$$p=(\langle f_0 \rangle {}^\frown \langle \dot{P}_{\dot{\beta_0}/\alpha_0},\dot{q}_0\rangle ) {}^\frown \cdots {}^\frown (\langle f_{n-1} \rangle {}^\frown \langle \dot{P}_{\dot{\beta}_{n-1}/\alpha_{n-1}},\dot{q}_{n-1}\rangle){}^\frown \langle g_0,\vec{g},A,F \rangle,$$
for some $n>0$, and a {\em common domain $d$} such that

\begin{enumerate}
    \item $(\langle f_0 \rangle {}^\frown \langle \dot{P}_{\dot{\beta_0}/\alpha_0},\dot{q}_0\rangle ) {}^\frown \cdots {}^\frown \langle f_{n-1} \rangle \in P_{\alpha_{n-1}}$, where $\alpha_{n-1}<\alpha$ (by the inductive construction, $\alpha_0<\cdots<\alpha_{n-1}$).
    \item $\Vdash_{\alpha_{n-1}} ``\alpha_{n-1} \leq \dot{\beta}_{n-1}<\alpha, \dot{q}_{n-1}\in \dot{P}_{\dot{\beta}_{n-1}/\alpha_{n-1}}"$.
    \item $d$ is an $\alpha$-domain (we emphasize that $d \in V$).
\item $A$ is a $d$-tree, $\min(A(\alpha))>\beta^*$, where $\beta^*=\sup\{\gamma \mid \exists r \in P_{\alpha_{n-1}}(r \Vdash \dot{\beta}_{n-1}=\gamma)\}$.
\item $\dom(F)=A(\alpha)$, and for each $\nu \in A(\alpha)$, $F(\nu)=\langle \dot{P}_{\dot{\beta}_\nu/\nu},\dot{q}\rangle$, where $\Vdash_{\nu} ``\nu \leq \dot{\beta}_\nu<\alpha \text{ and } \dot{q} \in \dot{P}_{\dot{\beta}_\nu/\nu}"$.
\item $\vec{g}=\{g_{\nu^\prime} \mid \nu^\prime \in A(\alpha)\}$.
\item $\dom(g_0)=d$ and for all $\nu^\prime$, $\dom(g_{\nu^\prime})=d$.
\item For $\zeta \in d$, $\Vdash_{P_{\alpha_{n-1}}*\dot{P}_{\dot{\beta}_{n-1}/\alpha_{n-1}}}``g_0(\zeta)<\alpha"$, and for all $\nu^\prime$, $\Vdash_{P_{\nu^\prime}*\dot{P}_{\dot{\beta}_{\nu^\prime}/\nu^\prime}}``g_{\nu^\prime}(\zeta)<\alpha"$.

\end{enumerate}

We write $p \restriction P_{\alpha_i}=(\langle f_0 \rangle {}^\frown \langle \dot{P}_{\dot{\beta_0}/\alpha_0},\dot{q}_0\rangle ) {}^\frown \cdots {}^\frown \langle f_i \rangle$, so $p \restriction P_{\alpha_i}\in P_{\alpha_i}$. Also write $p \restriction i=(\langle f_0 \rangle {}^\frown \langle \dot{P}_{\dot{\beta_0}/\alpha_0},\dot{q}_0\rangle ) {}^\frown \cdots {}^\frown (\langle f_i \rangle {}^\frown \langle \dot{P}_{\dot{\beta}_i/\alpha_i},\dot{q}_i\rangle)$, and we consider $p \restriction i$ as a condition in $P_{\alpha_i}*\dot{P}_{\dot{\beta}_i/\alpha_i}$. We put the superscript $p$ to every component, including the common domain, i.e. we write $d^p$ for $d$. We call $\dot{q}_i$'s the {\em interleaving part of $p$}.
With $p$ as above, we write $\tp(p)=\langle g_0,\vec{g},A,F \rangle$, $\stem(p)=p \setminus \tp(p)$ and say that $\stem(p)$ has $n$ blocks and write $n^p=n$. From the definition, it is straightforward to check that $|P_\alpha|=\alpha^{++}$.

\end{definition}

\begin{definition}[The one-step extension]
Let $$p=(\langle f_0 \rangle {}^\frown \langle \dot{P}_{\dot{\beta_0}/\alpha_0},\dot{q}_0\rangle ) {}^\frown \cdots {}^\frown (\langle f_{n-1} \rangle {}^\frown \langle \dot{P}_{\dot{\beta}_{n-1}/\alpha_{n-1}},\dot{q}_{n-1}\rangle){}^\frown \langle g_0,\vec{g},A,F \rangle,$$ with its common domain $d$, and $\langle \mu \rangle \in \Lev_0(A)$. Say $\nu=\mu(\alpha)$. The {\em one-step extension of $p$ by $\mu$}, denoted by $p+\langle \mu \rangle$, is the condition 
$$p^\prime=(\langle f_0 \rangle {}^\frown \langle \dot{P}_{\dot{\beta_0}/\alpha_0},\dot{q}_0\rangle ) {}^\frown \cdots {}^\frown (\langle f_{n-1} \rangle {}^\frown \langle \dot{P}_{\dot{\beta}_{n-1}/\alpha_{n-1}},\dot{q}_{n-1}\rangle){}^\frown r_0 {}^\frown r_1,$$
where
\begin{enumerate}
    \item $r_0=(g_0 \circ \mu^{-1},F(\nu))$,
    \begin{itemize}
        \item $g_0 \circ \mu^{-1}$ has domain $\rng(\mu)$.
        \item for $\zeta \in  \dom(\mu)$, $(g_0 \circ \mu^{-1})(\mu(\zeta))=g_0(\zeta)$.
        \item Write $F(\nu)=\langle \dot{P}_{\dot{\beta}_\nu/\nu}/\dot{q}\rangle$.
    \end{itemize}
    \item $r_1=\langle h_0^\prime,\vec{h}^\prime,A^\prime,F^\prime \rangle$,
    \begin{itemize}

    \item $A^\prime=\{\vec{\tau} \in A_{\langle \mu \rangle}\mid \tau_0(\alpha)>\beta^*\}$, where $\beta^*=\sup\{\gamma \mid \exists r \in P_\nu (r \Vdash_\nu ``\dot{\beta}_\nu=\gamma")\}$.
    \item $F^\prime=F \restriction A^\prime(\alpha)$.
    \item $\vec{h}=\{g_{\nu^\prime} \mid \nu^\prime \in A^\prime(\alpha)\}$.
    \item $\dom(h_0)=d$, and for all $\nu^\prime$, $\dom(h_{\nu^\prime})=d$.
    \item $ \Vdash_{P_\nu*\dot{P}_{\dot{\beta}/\nu}} ``h_0=g_\nu \oplus \mu"$, i.e. for $\zeta \in d$, if $\zeta \in \dom(\mu)$, $h_0(\zeta)=\check{\mu(\alpha)}$, otherwise, $\Vdash_{P_\nu*\dot{P}_{\dot{\beta}/\nu}} ``h_0(\zeta)=g_\nu(\zeta)"$ (we may assume that for $\zeta \in d \setminus \dom(\mu)$, $h_0(\zeta)=g_\nu(\zeta)$).
    \item for $\nu^\prime \in A^\prime(\alpha)$, $h_{\nu^\prime}=g_{\nu^\prime}$
    \end{itemize}
\end{enumerate}
    
\end{definition}

We define $p+\langle \rangle$ as $p$, and by recursion, define $p+\langle \mu_0, \cdots, \mu_n \rangle=(p+\langle \mu_0, \cdots, \mu_{n-1} \rangle)+\langle \mu_n \rangle$.

\begin{definition}[The direct extension relation]

Let $$p=(\langle f_0 \rangle {}^\frown \langle \dot{P}_{\dot{\beta_0}/\alpha_0},\dot{q}_0\rangle ) {}^\frown \cdots {}^\frown (\langle f_{n-1} \rangle {}^\frown \langle \dot{P}_{\dot{\beta}_{n-1}/\alpha_{n-1}},\dot{q}_{n-1}\rangle){}^\frown \langle g_0,\vec{g},A,F \rangle,$$
and
$$p^\prime=(\langle h_0 \rangle {}^\frown \langle \dot{P}_{\dot{\xi_0}/\gamma_0},\dot{r}_0\rangle ) {}^\frown \cdots {}^\frown (\langle h_{m-1} \rangle {}^\frown \langle \dot{P}_{\dot{\xi}_{m-1}/\gamma_{m-1}},\dot{r}_{m-1}\rangle){}^\frown \langle t_0,\vec{t},A^\prime,F^\prime \rangle.$$
We say that $p$ is a {\em direct extension of} $p^\prime$, denoted by $p \leq^*_\alpha p^\prime$, if the following hold.

\begin{enumerate}
    \item $n=m$.
    \item for $i<n$, $\alpha_i=\gamma_i$.
    \item $p \restriction n \leq^* p^\prime \restriction n$ in $P_{\alpha_{n-1}}* \dot{P}_{\dot{\beta}_{n-1}/\alpha_{n-1}}$, i.e. 
    \begin{itemize}
    \item $f_0 \supseteq h_0$.
    \item for $i \leq n$, $p \restriction P_{\alpha_i} \Vdash_{\alpha_i}`` \dot{\beta}_i=\dot{\xi}_i$ and $\dot{q}_i \leq^*_{\dot{P}_{\dot{\beta}_i/\alpha_i}} \dot{r}_i"$ (we can take $\dot{\beta}_i=\dot{\xi}_i$).
    \item for $i \in (0,n)$, $\dom(f_i) \supseteq \dom(h_i)$, and for $\zeta \in \dom(h_i)$, $p \restriction i \Vdash_{P_{\alpha_i}*\dot{P}_{\dot{\beta}_i/\alpha_i}} ``f_i(\zeta)=h_i(\zeta)"$.
    \end{itemize}
    \item $d^p \supseteq d^{p^\prime}$.
    \item $A \restriction d^{p^\prime} \subseteq A^\prime$.
    \item for every $\nu \in A(\alpha)$  and $\vec{\mu} \in A$ with $\vec{\mu}(\alpha)=\nu$,
    $$p+\vec{\mu} \restriction P_\nu \Vdash_\nu ``F(\nu)_0=F^\prime(\nu)_0 \text{ and } F(\nu)_1 \leq^*_{F(\nu)_0} F^\prime(\nu)_1".$$
    \item For $\zeta \in d^{p^\prime}$,
    \begin{itemize}
        \item $p \restriction n \Vdash_{P_{\alpha_{n-1}}*\dot{P}_{\dot{\beta}_{n-1}/\alpha_{n-1}}} ``g_0(\zeta)=t_0(\zeta)"$.
        \item for $\nu \in A(\alpha)$, write $F(\nu)=(\dot{P}_{\dot{\beta}_\nu/\nu},\dot{q})$, and every $\vec{\mu}$ with $\vec{\mu}(\alpha)=\nu$, we have
        $$p+\vec{\mu} \restriction (n+|\vec{\mu}|) \Vdash_{P_\nu* \dot{P}_{\dot{\beta}_\nu/\nu}}``g_\nu(\zeta)=t_\nu(\zeta)".$$
    \end{itemize}
\end{enumerate}

\end{definition}

\begin{definition}[The extension relation]
\label{extrel}
Let $$p=(\langle f_0 \rangle {}^\frown \langle \dot{P}_{\dot{\beta_0}/\alpha_0},\dot{q}_0\rangle ) {}^\frown \cdots {}^\frown (\langle f_{n-1} \rangle {}^\frown \langle \dot{P}_{\dot{\beta}_{n-1}/\alpha_{n-1}},\dot{q}_{n-1}\rangle){}^\frown \langle g_0,\vec{g},A,F \rangle,$$
and $p^\prime \in P_\alpha$.
We say that $p$ is an {\em extension of} $p^\prime$, denoted by $p \leq_\alpha p^\prime$, if there is $\vec{\mu} \in A^{p^\prime}$, or $\vec{\mu}=\langle \rangle$, such that by letting $p^*=p^\prime+\vec{\mu}$ and write
$$p^*=(\langle h_0 \rangle {}^\frown \langle \dot{P}_{\dot{\xi_0}/\gamma_0},\dot{r}_0\rangle ) {}^\frown \cdots {}^\frown (\langle h_{m-1} \rangle {}^\frown \langle \dot{P}_{\dot{\xi}_{m-1}/\gamma_{m-1}},\dot{r}_{m-1}\rangle){}^\frown \langle t_0,\vec{t},A^\prime,F^\prime \rangle,$$
we then have that
\begin{enumerate}

\item $p \restriction n \leq p^* \restriction m$ in $P_{\alpha_{n-1}}*\dot{P}_{{\beta}_{n-1}/\alpha_{n-1}}$, namely,
\begin{itemize}
    \item $\alpha_{n-1}=\gamma_{m-1}$.
    \item $p \restriction P_{\alpha_{n-1}} \leq_{\alpha_{n-1}} p^* \restriction P_{\alpha_{n-1}}$.
    \item $p \restriction P_{\alpha_{n-1}} \Vdash_{\alpha_{n-1}} ``\dot{\beta}_{n-1}=\dot{\gamma}_{m-1}$ and $\dot{q} \leq_{\dot{P}_{\dot{\beta}_{n-1}/\alpha_{n-1}}} \dot{r}_{m-1}"$ (we can take $\dot{\beta}_{n-1}=\dot{\gamma}_{m-1})$.
    
\end{itemize}
\item $d^p \supseteq d^{p^*}$.
    \item $A \restriction d^{p^*} \subseteq A^\prime$.
    \item for every $\nu \in A(\alpha)$  and $\vec{\mu} \in A$ with $\vec{\mu}(\alpha)=\nu$,
    $$p+\vec{\mu} \restriction P_\nu \Vdash_\nu ``F(\nu)_0=F^\prime(\nu)_0 \text{ and } F(\nu)_1 \leq^*_{F(\nu)_0} F^\prime(\nu)_1".$$
    (the $\leq^*$ here is intentional).
    \item For $\zeta \in d^{p^*}$,
    \begin{itemize}
        \item $p \restriction n \Vdash_{P_{\alpha_{n-1}}*\dot{P}_{\dot{\beta}_{n-1}/\alpha_{n-1}}} ``g_0(\zeta)=t_0(\zeta)"$.
        \item for $\nu \in A(\alpha)$, write $F(\nu)=\langle \dot{P}_{\dot{\beta}/\nu},\dot{q}\rangle$, then 
        $$p+\vec{\mu} \restriction (n+|\vec{\mu}|) \Vdash_{P_\nu* \dot{P}_{\dot{\beta}/\nu}}``g_\nu(\zeta)=t_\nu(\zeta)".$$
    \end{itemize}
\end{enumerate}

\end{definition}

\begin{remark}
 In Definition \ref{extrel}, $n=m$. This is because $\alpha$ is the first cardinal with $\circ(\alpha)>0$.  
\end{remark}

Note that equivalently, $p \leq p^\prime$ if there is $\vec{\mu}$ such that $p$ is a condition obtained by extending the interleaving part of a direct extension of $p^\prime+\vec{\mu}$. For $p^\prime \leq p$, the {\em interpolant of $p^\prime$ and $p$} is $p^*$ such that there exist unique $\vec{\mu}$ such that $p^*=p+\vec{\mu}$ and $p^\prime$ is obtained by extending the interleaving part of the direct extension of $p^*$.

\begin{proposition}
$(P_\alpha,\leq)$ has the $\alpha^{++}$-chain condition.
\label{cc1}
\end{proposition}

\begin{proof}
Let $\{ p^\gamma \mid \gamma<\alpha^{++}\}$ be a collection of conditions in $P_\alpha$. $p_\gamma$ can be written as $p^\gamma_0 {}^\frown \langle f^\gamma_0,\vec{f}^\gamma,A^\gamma,F^\gamma \rangle$, with the corresponding common domain $d^\gamma$.
By shrinking the collection, we may assume that there are $p_0,d,b$ such that for all $\gamma$, $p^\gamma_0=p_0$, $b=A^\gamma(\alpha)$, and $d$ is the root of the $\Delta$-system $\{d^\gamma \mid \gamma<\alpha^{++}\}$.  Since for each $\gamma<\alpha^{++}$, $\zeta \in d$, and $\nu \in b$, $f^\gamma_0(\zeta)$, $f^\gamma_\nu(\zeta) \in V_\alpha$, and $F^\gamma(\nu) \in V_\alpha$, we can shrink the collection of conditions further so that there are $x_{\zeta,0}, x_{\zeta,\nu},y_\nu$, such that for all $\gamma<\alpha^{++}$, $f^\gamma_0(\zeta)=x_{\zeta,0}$, $f^\gamma_\nu(\zeta)=x_{\zeta,\nu}$, and $F^\gamma(\nu)=y_\nu$. Then any two conditions in the shrunk collection are compatible.

\end{proof}

\begin{proposition}
$(\{p \in P_\alpha \mid p \text{ is pure}\},\leq^*)$ is $\alpha$-closed.
\label{closure1}
\end{proposition}

\begin{proof}
    Let $\beta<\alpha$ and $\langle p^{\beta^\prime} \mid \beta^\prime<\beta\rangle$ be a $\leq^*$-decreasing sequence of conditions in $P_\alpha$.
    Write $p^{\beta^\prime}=\langle f_0^{\beta^\prime},\vec{f}^{\beta^\prime},A^{\beta^\prime},F^{\beta^\prime} \rangle$ with its common domain $d^{\beta^\prime}$.
    Let $d^*=\cup_{\beta^\prime<\beta} d^{\beta^\prime}$, $f^*_0=\cup_{\beta^\prime<\beta} f_0^{\beta^\prime}$. Let $(A^{\beta^\prime})^*$ be the $d^*$-tree obtained by pulling back $A^{\beta^\prime}$, and $A^*=\cap_{\beta^\prime<\beta} (A^{\beta^\prime})^*$. Shrink $A^*$ further so that $\min(A^*(\alpha))>\beta$. By induction on $\nu \in A^*(\alpha)$, we may find $f^*_\nu$ and $F^*(\nu)$ such that
    \begin{itemize}
        \item for $\zeta \in d^*$, $f^*_\nu(\zeta)$ is ``forced" to be equal to $f^{\beta^\prime}_\nu(\zeta)$ for some sufficiently large $\beta^\prime$ that $\zeta \in \dom(f^{\beta^\prime})$.
        \item $F^*(\nu)=\langle \dot{P}_{\dot{\beta}_\nu/\nu},\dot{q}_\nu^*\rangle$ is such that $\dot{q}_\nu^*$ is ``forced" to be a $\leq^*$-lower bound of $\langle \dot{q}^{\beta^\prime}_\nu \mid \beta^\prime<\beta \rangle$, where $F^{\beta^\prime}(\nu)=\langle \dot{P}_{\dot{\beta}_\nu/\nu},\dot{q}^{\beta^\prime}_\nu\rangle$. This is possible because $\Vdash_\nu ``(\dot{P}_{\dot{\beta}_\nu/\nu},\leq^*) \text{ is $\nu^*$-closed}"$, where $\nu^*$ is the least inaccessible above $\nu$, and $\nu>\beta$.
    \end{itemize}
    Then $\langle f^*,\vec{f}^*,A^*,F^* \rangle$, where $\vec{f}^*=\{ f^*_\nu \mid \nu \in A^*(\alpha) \text{ is inaccessible}\}$, is as required.
 \end{proof}   

\begin{theorem}
\label{prikry1}
$(P_\alpha,\leq,\leq^*)$ has the Prikry property, i.e. for $p \in P_\alpha$ and a forcing statement $\varphi$, there is $p^* \leq^* p$ such that $p^* \parallel \varphi$.
\end{theorem}

To prove Theorem \ref{prikry1}, we start with the following lemma.

\begin{lemma}
\label{maximize1}
Let $p \in P_\alpha$ and $\varphi$ be a forcing statement. Then there is $p^* \leq^* p$ such that if $r=r_0 {}^\frown \tp(r)$, $r \leq p^*$, $r \parallel \varphi$, and $p^\prime$ is the interpolant of $r$ and $p^*$, then

$$r_0  {}^\frown \tp(p^\prime) \parallel \varphi \text{ the same way}.$$
\end{lemma}

\begin{proof}
    Assume for simplicity that $p$ is pure and write $p=\langle f_0,\vec{f},A,F\rangle$ with its common domain $d$. A forcing $\mathbb{A}$ consists of conditions of the form $g=\langle g_0 \rangle {}^\frown \vec{g}$, where there is a common domain $d_g$ such that
    \begin{itemize}
        \item $\dom(g_0)=d_g$, $\vec{g}=\langle g_\nu \mid \nu \in A(\alpha) \rangle$, and for all $\nu$, $\dom(g_\nu)=d_g$.
        \item for $\zeta \in d_g$, $f_0(\zeta)<\alpha$ and for $\beta<\alpha$ inaccessible, $\Vdash_{P_\nu *\dot{P}_{\dot{\beta}_\nu/\nu}} ``g_\beta(\zeta)<\alpha"$.
    \end{itemize}

    For $g^0,g^1 \in \mathbb{A}$, define $g^0 \leq_{\mathbb{A}} g^1$ if $g^0_0 \supseteq g^1_0$, and for $\nu \in A$ and $\zeta \in d_{g^1}$, $\Vdash_{P_\nu*\dot{P}_{\dot{\beta}_\nu/\nu}} ``g^0_\nu(\zeta)=g^1_\nu(\zeta)"$.
    Clearly, $\mathbb{A}$ is $\alpha^+$-closed.

    Let $N \prec H_\theta$ for some sufficiently large regular $\theta$, ${}^{<\alpha}N \subseteq N$, $|N|=\alpha$, $d,V_\alpha \subseteq N$, $p,\mathbb{P},\mathbb{A} \in N$. Build an $\mathbb{A}$-decreasing sequence $\langle f^\gamma \mid \gamma<\alpha \rangle$ below $\langle f_0 \rangle {}^\frown \vec{f}$ such that for every dense open set $D \in N \cap \mathcal{P}(\mathbb{A})$, there are unboundedly many $\gamma<\alpha$ such that $f^\gamma \in D$.
    Let $f^*=\langle f^*_0 \rangle {}^\frown \vec{f}^*$ be the maximal $\leq^*$-lower bound of $\langle f^\gamma \mid \gamma<\alpha \rangle$ and $d^*$ be its common domain, so $d^*=N \cap \alpha^{++}$. Let $A^*$ be the $d^*$-tree which is the pullback of $A$. Note that $A^* \subseteq N$. We may assume $A^*$ is generated by $B^* \subseteq \mathcal{B}_{d^*}$.

    We are now going to consider an $\mathbb{A}$-decreasing subsequence $\langle f^{\gamma_\nu} \mid \nu \in A^*(\alpha) \rangle$ of $\langle f^\gamma \mid \gamma<\alpha \rangle$, together with $\langle \dot{q}^\nu_{\nu^\prime} \mid \nu,\nu^\prime \in A^*(\alpha) \rangle$ and $\langle A^\nu \mid \nu \in A^*(\alpha) \rangle$ which satisfy a certain property, and
    \begin{itemize}
        \item for each $\nu^\prime$, $\langle \dot{q}^\nu_{\nu^\prime} \mid \nu \in A^*(\alpha) \rangle$ is forced to be $\leq^*$-decreasing below $\dot{q}_{\nu^\prime}$, where $F(\nu^\prime)=\langle \dot{P}_{\dot{\beta}_{\nu^\prime}/\nu^\prime},\dot{q}_{\nu^\prime} \rangle$.
        \item for $\nu^\prime<\nu$, $\dot{q}^\nu_{\nu^\prime}=\dot{q}^{\nu^\prime}_{\nu^\prime}$.
    \end{itemize}
    All the proper initial subsequences will be in $N$ (the key point is that ${}^{<\alpha}N \subseteq N$.
    Let $\nu \in A^*(\alpha)$ and suppose that $\langle f^{\gamma_{\nu^\prime}} \mid \nu^\prime<\nu, \nu^\prime \in A^*(\alpha) \rangle$, $\langle \dot{q}^{\nu^\prime}_{\rho} \mid \nu^\prime<\nu$, $\nu^\prime,\rho \in A^*(\alpha)\rangle$, and $\langle A^{\nu^\prime} \mid \nu^\prime<\nu$, $\nu^\prime \in A^*(\alpha) \rangle$ have been constructed. 
    For $\nu^\prime<\nu$, let $\dot{q}^\nu_{\nu^\prime}=\dot{q}^{\nu^\prime}_{\nu^\prime}$.
    Let $f^\prime$ be the maximal lower bound of the sequence $\langle f^{\gamma_{\nu^\prime}} \mid \nu^\prime<\nu, \nu^\prime \in A^*(\alpha) \rangle$.  For $\rho \geq \nu$, Let $\dot{q}_\rho^*$ be a $P_\rho$-name of a condition in $\dot{P}_{\dot{\beta}_\rho/\rho}$ which is forced to be a $\leq^*$-maximal lower bound of $(\dot{q}^{\nu^\prime}_\rho)_{\nu^\prime<\nu}$. This is possible since $\Vdash_\rho ``(\dot{P}_{\dot{\beta}_\rho/\rho},\leq^*) \text{ is } \nu^+\text{-closed}"$ and note that $\langle \dot{q}^*_\rho \mid \rho \geq \nu \rangle \in N$.
    Consider the following set $D_\nu \subseteq \mathbb{A}$. $g=\langle g_0 \rangle {}^\frown \vec{g} \in D_\nu$ with the common domain $d_g$, if either $\langle g_0 \rangle {}^\frown \vec{g}$ is incompatible with $\langle f_0 \rangle {}^\frown \vec{f}$, or the following holds:

    \begin{itemize}
    \item for every $\vec{\mu} \in A^*$ with $\vec{\mu}(\alpha)=\nu$, $\dom(\vec{\mu}) \subseteq d_g$.
    \item there are 
    \begin{itemize}

        \item a $d_g$-tree $A^\nu$ with $\min(A^\nu(\alpha)) > \xi^*:=\{\xi \mid \exists t \in P_\nu (t \Vdash_\nu ``\dot{\beta}_\nu=\xi")\}$, and
        \item a function $F^\nu$ with $\dom(F^\nu)=A^\nu(\alpha)$,
        \item for $\rho \in A^\nu(\alpha)$, $ \Vdash_\rho ``F^\nu(\rho)_1 \leq^* \dot{q}^*_\rho"$,
    \end{itemize}
    such that for every $r \in P_\nu* \dot{P}_{\dot{\beta}_\nu/\nu}$, if there are $h_0, \vec{h}$, $A^\prime$, and $F^\prime$ such that
    $$r {}^\frown \langle h_0,\vec{h}, A^\prime,F^\prime \rangle \leq^* r {}^\frown  \langle g_{\nu},\langle g_{\nu^\prime} \mid \nu^\prime \in A^\nu(\alpha)\rangle,A^\nu,F^\nu \rangle,$$
    and 
    $$r {}^\frown \langle h_0,\vec{h}, A^\prime,F^\prime \rangle \parallel \varphi,$$
    then $$r {}^\frown \langle g_{\nu},\langle g_{\dot{\beta}_{\nu^\prime}} \mid \nu^\prime \in A^\nu(\alpha)\rangle,A^\nu,F^\nu \rangle \parallel \varphi \text{ the same way.}$$
    \end{itemize}

\begin{claim}
\label{claimmax1}
$D_\nu \in N$ is open dense.
\end{claim}
\begin{proof}
The parameters we use to define $D_\nu$ are: $\mathbb{A}$, $p$, and $P_\nu* \dot{P}_{\dot{\beta}_\nu/\nu}$. Thus, $D_\nu \in N$.
To check the openness of $D_\nu$, note that if $\vec{g}^0 \leq_{\mathbb{A}} \vec{g}^1$ and $\vec{g}^1 \in D_\nu$ with the witnesses $A^\nu$. and $F^\nu$, then $\vec{g}^0$ is also in $D_\nu$ with the same witnesses. 

It remains to show that $D_\nu$ is dense. Let $g_0 {}^\frown \vec{g} \in \mathbb{A}$. If $\langle g_0 \rangle {}^\frown \vec{g} \nparallel \langle f_0 \rangle {}^\frown \vec{f}$, then we are done. Suppose not, we may assume $\langle g_0 \rangle {}^\frown \vec{g} \leq_{\mathbb{A}} \langle f_0 \rangle {}^\frown \vec{f}$. By  Proposition \ref{indsch} for $\nu$, let $\langle r_\xi \mid \xi<(\xi^*)^{++} \rangle$ be an enumeration of elements in $P_\nu*\dot{P}_{\dot{\beta}_\nu/\nu}$ (with some repetitions if needed). Build sequences $\langle \langle h_0^\xi \rangle {}^\frown \vec{h}^\xi \rangle$, $\langle A_\xi, F_\xi \mid \xi \leq (\xi^*)^{++} \rangle$ such that
\begin{itemize}
    \item $\langle \langle h_0^\xi \rangle {}^\frown \vec{h}^\xi \rangle_{\xi \leq (\xi^*)^{++}}$ is $\mathbb{A}$-decreasing, and is below $\langle g_0 \rangle{}^\frown \vec{g}$.
    \item $\langle A_\xi \mid \xi \leq (\xi^*)^{++} \rangle$ is a $\dom(h^\xi_0)$-tree, for each $\xi$, $A_\xi$ is a $\dom(h_0^\xi)$-tree, $\min(A_\xi(\alpha))>\xi^*$, and for $\xi<\xi^\prime$, $A^{\xi^\prime}$ projects down to a subset of $A^\xi$.
    \item for $\nu^\prime \in A_\xi(\alpha)$, $\langle F_\xi(\nu^\prime)_1 \rangle_{\xi \leq (\xi^*)^{++}}$ is forced to be $\leq^*$-decreasing below $\dot{q}_{\nu^\prime}^*$.
    \item for $\xi<(\xi^*)^{++}$, if there are $h_0^\prime, \vec{h}^\prime, A^\prime$, and $F^\prime$ such that
    $$r_\xi {}^\frown \langle h_0^\prime, \vec{h}^\prime, A^\prime,F^\prime \rangle$$ is a direct extension of  $$ t^*:=r_\xi {}^\frown  \langle h^{\xi+1}_{\nu}, \langle h^{\xi+1}_{\rho} \mid \rho \in A_{\xi+1}(\alpha) \rangle,A_{\xi+1},F_{\xi+1} \rangle,$$
    and $$r_\xi {}^\frown \langle h_0^\prime, \vec{h}^\prime, A^\prime,F^\prime \rangle \Vdash \varphi,$$
    then $t^*$ decides $\varphi$ the same way.
    \end{itemize}
    The construction is straightforward, and for a limit $\xi$, we can take the obvious $A_\xi$ and $F_\xi$ which satisfy the requirement. Finally, let $\langle g_0 \rangle {}^\frown \vec{g}= \langle h_0^{(\xi^*)^{++}} \rangle {}^\frown \vec{h}^{(\xi^*)^{++}}$, $A^\nu=A_{(\xi^*)^{++}}$, and $F^\nu=F_{(\xi^*)^{++}}$. These will be the witnesses for $\langle g_0 \rangle {}^\frown \vec{g} \in D_\nu$.

\end{proof}

Let $\gamma_\nu \geq \sup_{\nu^\prime<\nu} \gamma_{\nu^\prime}$ such that $f^{\gamma_\nu} \in D_\nu$. Also, we obtain the witnesses, $A^\nu$ and $F^\nu$. Let $\dot{q}^\nu_\nu=\dot{q}^*_\nu$.
For $\rho>\xi^*$, let $\dot{q}^\nu_\rho=F^\nu(\rho)_1$ if exists, otherwise, let $\dot{q}^\nu_\rho=\dot{q}_\rho^*$. We also take $\dot{q}^\nu_\rho=\dot{q}_\rho^*$ for other $\rho$ where $\dot{q}^\nu_\rho$ is not yet defined.
This completes our analysis.

Assume that the pullback of $A^\nu$ to the $d^*$-tree has a subtree which is generated by $B^\nu \in E(d^*)$. Let $A^{**}$ be a $d^*$-tree generated by $\Delta_\nu B^\nu$. Let $F^{**}$ be a function with $\dom(F^{**})=A^{**}(\alpha)$ and for $\nu \in A^{**}(\alpha)$, $F^{**}(\nu)=\langle \dot{P}_{\dot{\beta}_\nu/\nu},\dot{q}^{**}_\nu \rangle$, $\dot{q}^{**}_\nu$ is the $\leq^*$-maximal lower bound of $(\dot{q}^{\nu^\prime}_\nu)_{\nu^\prime \in A(\alpha)}$. This is possible since $(\dot{q}^{\nu^\prime}_\nu)_{\nu^\prime \in A(\alpha)}$ stabilizes after the stage $\nu^\prime=\nu$ (equivalently, we take $\dot{q}^{**}_\nu=\dot{q}^\nu_\nu$).
Then $p^*=\langle f_0^*, \vec{f}^*,A^{**},F^{**} \rangle \leq^* p$.

We now show that $p^*$ satisfies Lemma \ref{maximize1}. Let $p^\prime \leq p^*$ such that $p^\prime$ decides $\varphi$, $p^\prime$ is of the form 
$$p^\prime=r {}^\frown \langle h_0^\prime,\vec{h}^\prime,A^\prime,F^\prime \rangle,$$

Without loss of generality, assume that $p^\prime \Vdash \varphi$. Let $\bar{p}$ be the interpolant of $p^*$ and $p^\prime$.
We consider the notions of the proof of Claim \ref{claimmax1}. Say that $r=r_\xi$. By the construction of $A^{**}$, we have that $A^{**}$ projects down to a subset of $A^\nu$. This makes $p^\prime \leq^* t^*$, and hence, $t^* \Vdash \varphi$.
Thus, $r_\xi{}^\frown  \tp(\bar{p}) \Vdash \varphi$. This completes the proof of Lemma \ref{maximize1}.
    
\end{proof}

\begin{proof}[Proof of Theorem \ref{prikry1}]
Let $p$ be a condition and $\varphi$ be a forcing statement. For simplicity, assume $p$ is pure and $p$ satisfies Lemma \ref{maximize1}. Write $p=\langle f_0,\vec{f},A,F \rangle$, $d$ is the common domain for $p$. Assume $A$ is generated by $B \subseteq \mathcal{B}_d$. By shrinking $B$, assume that for $\nu^\prime<\nu$, $\Vdash_{\nu^\prime} ``\dot{\beta}_{\nu^\prime}<\nu"$. 

 By Remark \ref{ppoitre}, let $\{ \mu_\gamma \mid \gamma<\nu^{++}\}$ be an enumeration of $\mu \in B$ and $\mu(\alpha)=\nu$. Let $\{r_\xi \mid \xi<\nu\}$ be an enumeration of  $r \in \cup_{\nu^\prime<\nu}P_{\nu^\prime}*\dot{P}_{\dot{\beta}_{\nu^\prime}/\nu^\prime}$. We are only interested in pairs $r_\xi,\mu_\gamma$ such that $r_\xi \leq \stem(p+\vec{\tau}_\xi)$ for some $\vec{\tau}_\xi<\mu_\gamma$.
Build $\langle \dot{q}^{\gamma,\xi} \mid \gamma<\nu^{++}, \xi <\nu \rangle$ and $\langle f^{\gamma,\xi} \mid \gamma<\nu^{++}, \xi<\gamma \rangle$ such that 
\begin{enumerate}
    \item $(\gamma,\xi)<(\gamma^\prime,\xi^\prime)$ implies $\Vdash_\nu ``\dot{q}^{\gamma^\prime,\xi^\prime} \leq^* \dot{q}^{\gamma,\xi}"$.
    \item for each $\gamma$, and $\xi<\xi^\prime$ such that $r_\xi$ and $r_{\xi^\prime}$ are in the same forcing $P_{\nu^\prime}* \dot{P}_{\dot{\beta}_{\nu^\prime}/\nu^\prime}$, we have that $\Vdash_{P_{\nu^\prime}* \dot{P}_{\dot{\beta}_{\nu^\prime}/\nu^\prime}} ``f^{\gamma,\xi^\prime} \leq f^{\gamma,\xi} \leq f_{\nu^\prime} \circ \mu_\gamma^{-1}"$.
    \item there is $r_{\gamma,\xi}^* \leq^* r_\xi$ such that $r_{\gamma,\xi}^* {}^\frown (f^{\gamma,\xi},\langle \dot{P}_{\dot{\beta}_\nu/\nu},\dot{q}^{\gamma,\xi+1} \rangle) \leq^* r_\xi {}^\frown (f_{\vec{\tau}_\xi(\alpha)} \circ \mu_\gamma^{-1}, \langle \dot{P}_{\dot{\beta}_\nu/\nu},\dot{q}^{\gamma,\xi}\rangle),$ and $r_\xi^* {}^\frown (f^{\xi,\gamma},\langle \dot{P}_{\dot{\beta}_\nu/\nu},\dot{q}^{\gamma,\xi+1} \rangle)$ decides $\varphi_{\vec{\tau}_\xi,\mu_\xi}$, where $\dot{G}_\nu$ is the canonical name for generic over $P_\nu* \dot{P}_{\dot{\beta}_\nu/\nu}$, and
$$\varphi_{\vec{\tau}_\xi,\mu_\xi} \equiv \exists t \in \dot{G}_\nu(t {}^\frown \tp(p+\vec{\tau}_\xi {}^\frown \langle \mu_\xi \rangle) \parallel \varphi).$$
\end{enumerate}
This is possible by the Prikry property at the level below $\alpha$.
By extending further, we assume that
$r_\xi^* {}^\frown (f^{\gamma,\xi},\langle \dot{P}_{\dot{\beta}_\nu/\nu},\dot{q}^{\gamma,\xi+1} \rangle) \Vdash \varphi_{\vec{\tau}_\xi,\mu_\xi}^i$ for unique $i \in \{0,1,2\}$, where

\begin{align*}
    \varphi_{\vec{\tau}_\xi,\mu_\xi}^0 &\equiv \exists t \in \dot{G}_\nu(t {}^\frown \tp(p+\vec{\tau}_\xi {}^\frown \langle \mu_\xi \rangle) \Vdash \varphi), \\
    \varphi_{\vec{\tau}_\xi,\mu_\xi}^1 & \equiv \exists t \in \dot{G}_\nu(t {}^\frown \tp(p+\vec{\tau}_\xi {}^\frown \langle \mu_\xi \rangle) \Vdash \neg \varphi),\\
     \varphi_{\vec{\tau}_\xi,\mu_\xi}^2 & \equiv \nexists t \in \dot{G}_\nu(t {}^\frown \tp(p+\vec{\tau}_\xi {}^\frown \langle \mu_\xi \rangle) \parallel  \varphi).
\end{align*}
For limit $\xi>0$, take $\dot{q}^{\gamma,\xi}$ which is forced to be a $\leq^*$-lower bound of $\dot{q}^{\gamma,\xi^\prime}$ for $\xi^\prime<\xi$. Also for $\gamma>0$, take $\dot{q}^{\gamma,0}$ as a $\leq^*$-lower bound of $\{\dot{q}^{\gamma^\prime,\xi} \mid \gamma^\prime<\gamma$ and all $\xi \}$. The construction proceeds since the following the following hold:
\begin{itemize}
    \item $\Vdash_\nu "(\dot{P}_{\dot{\beta}_\nu/\nu},\leq^*)$ is $\nu^*$-closed", where $\nu^*$ is the first inaccessible cardinal greater than $\nu$. In particular, it is $\nu^{+3}$-closed.
    \item For $\nu^\prime<\nu$, $\Vdash_{P_{\nu^\prime},\dot{P}_{\dot{\beta}_{\nu^\prime},\nu^\prime}} \dot{C}(\nu^+,\nu^{++})$ is $\nu^+$-closed. 
\end{itemize} 

Finally, let $\dot{q}_\nu^*$ be the maximal $\leq^*$-lower bound of $\dot{q}^{\gamma,\xi}$.

For each $\mu=\mu_\gamma$ and $\nu^\prime<\nu$, consider the family $\mathcal{F}=\{f^{\gamma,\xi} \mid \vec{\tau}_\xi(\alpha)=\nu^\prime\}$. This is forced to be $\leq$-decreasing in $\dot{C}(\nu^+,\nu^{++})$ (in $V^{P_{\nu^\prime}*\dot{P}_{\dot{\beta}_{\nu^\prime}/\nu^\prime}}$) below $f_{\nu^\prime} \circ \mu^{-1}$. Let $f_{\nu^\prime}^\mu$ be the maximal $\leq$-lower bound of $\mathcal{F}$. Fix $\mu$. We can extend further each $f_{\nu^\prime}^\mu$ so that for $\nu^\prime_0,\nu^\prime_1$, $\dom(f_{\nu^\prime_0}^\mu)=\dom(f_{\nu^\prime_1}^\mu)$. Let $\mathcal{F}_\mu=\{f_{\nu^\prime}^\mu \mid \nu^\prime<\nu ,\nu^\prime \in A(\alpha) \cup \{0\}\}$ where 
\begin{itemize}
    \item for each $\mu$, $f_{\nu^\prime}^\mu$ is forced to be stronger than $f_{\nu^\prime} \circ \mu^{-1}$.
    \item there is $d_\mu$ such that for all $\nu^\prime$, $\dom(f_{\nu^\prime}^\mu)=d_\mu$.
\end{itemize}

Consider $\mathcal{G}=j(\mu \mapsto \mathcal{F}_\mu)(\mc(d))$. Then $\mathcal{G}=\langle f_0^* \rangle {}^\frown \langle f^*_{\nu^\prime} \mid \nu^\prime \in A(\alpha) \rangle$, with some common domain $d^*$. Note that the collection $B^*$ of $\psi \in \OB_{\alpha,0}(d^*)$ such that $\psi \restriction d \in B$, and for $\nu^\prime<\psi(\alpha)$ including $0$, $\Vdash_{P_{\nu^\prime} \circ \dot{P}_{\dot{\beta}_{\nu^\prime}/\nu^\prime}} f_{\nu^\prime}^* \circ \psi^{-1} \leq f^{\psi \restriction d}_{\nu^\prime}$" is of measure-one. Let $A^*$ be generated by $B^*$. Now let $F^*$ be such that $\dom(F^*)=A^*(\alpha)$ and for each $\nu$, $F^*(\nu)_1=\dot{q}_\nu^*$. Let $p_1=\langle f_0^*,\vec{f}^*,A^*,F^* \rangle$.

\begin{claim}
    Let $\bar{p}=p_1+\vec{\pi} {}^\frown \langle \psi \rangle$, $\vec{\tau}=\vec{\pi} \restriction d$, $\mu=\psi \restriction d$. Let $\nu^\prime=\vec{\tau}(\alpha)$ and $\nu=\mu(\alpha)$. Then $r^* {}^\frown (f_{\nu^\prime}^* \circ \psi^{-1},F^*(\nu)) \rangle$ decides $\varphi_{\vec{\tau},\mu}^i$ for some unique $i$.
\end{claim}

\begin{proof}
Write $r=r_\xi$ and $\mu=\mu_\gamma$. Note that $\Vdash f_{\nu^\prime}^* \circ \psi^{-1} \leq f_{\nu^\prime}^\mu \leq f^{\gamma,\xi}$ and $\Vdash F^*(\nu)_1 =\dot{q}_\nu^* \leq^* \dot{q}^{\gamma,\xi+1}$, hence, we are done.
\end{proof}

For each $r$ such that $r \leq \stem(p_1+\vec{\pi})$ and $\vec{\tau}=\vec{\pi} \restriction d$, we indicate $\vec{\tau}=\vec{\tau}_r$.
For $i<3$, let $B_{r,i}=\{\psi \in B^* \mid r^* {}^\frown \langle f_{\nu^\prime}^* \circ \psi^{-1}, F^*(\psi(\alpha))\rangle) \Vdash \varphi_{\vec{\tau}_r, \psi \restriction d}^i\}$. For each $r$, let $i(r)<3$ be unique such that $B_{r,i(r)}$ is of measure-one. 
Let $B_\nu=\cap \{B_{r,i(r)} \mid r \in P_\nu * \dot{P}_{\dot{\beta}_\nu/\nu}\}$ and $B^{**}=\Delta_\nu B_\nu$. Let $A^{**}$ be the $d^*$-tree generated by $B^{**}$ and $F^{**}=F^* \restriction B^*$.
Let $p^*=\langle f_0^*,\vec{f}^*,A^{**},F^{**} \rangle$.

\begin{claim}
$p^*$ satisfies the Prikry property.
\end{claim}

\begin{proof}
Let $p^\prime \leq p^*$ with $p^\prime \parallel \varphi$, Assume $p^\prime \Vdash \varphi$ and the interpolant of $p^\prime$, $p^*$, say $\bar{p}$, is such that $\bar{p}=p^*+\vec{\mu}$ with the minimal $n^*=|\vec{\mu}|$. If $n^*=0$. then we might apply $p^\prime$ for the Prikry property instead. Assume $n^*>0$. 

For simplicity, we establish the case $n^*=2$. Say $\bar{p}=p^*+ \langle \pi,\psi \rangle$. Write $\tau=\pi \restriction d$, $\mu= \psi \restriction d$.
Let $$p^\prime=(g_0, \langle \dot{P}_{\dot{\beta}_0/\nu_0},\dot{q}_0\rangle) {}^\frown (g_1, \langle \dot{P}_{\dot{\beta}_1/\nu_1},\dot{q}_1\rangle) {}^\frown \tp(p^\prime).$$

Since $p$ satisfies Lemma \ref{maximize1}, we have that 
$$(g_0, \langle \dot{P}_{\dot{\beta}_0/\nu_0},\dot{q}_0\rangle) {}^\frown (g_1, \langle \dot{P}_{\dot{\beta}_1/\nu_1},\dot{q}_1\rangle) {}^\frown \tp(\bar{p}) \Vdash \varphi.$$

Set $r=(g_0, \langle \dot{P}_{\dot{\beta}_0/\nu_0},\dot{q}_0\rangle)$. Recall that $r^* {}^\frown \langle f_{\nu_0}^* \circ \psi^{-1},F^{**}(\nu_1)) \Vdash \varphi_{\tau,\mu}^i$ for a unique $i$, so $i(r)$ exists.
We claim that $i(r)=0$. Otherwise, we may assume $i_r=1$ (the case $i_r=2$ is similar). Let $G$ be generic containing $r^* {}^\frown \langle g_1,\langle \dot{P}_{\dot{\beta}_1/\nu_1},\dot{q}_1 \rangle)$. Then there is $t \in G$ such that $t {}^\frown \tp(\bar{p}) \Vdash \neg \varphi$, but if $t \leq r^* {}^\frown (g_1, \langle \dot{P}_{\dot{\beta}_1/\nu_1},\dot{q}_1 \rangle)$, we get a condition having contradictory decisions, which is a contradiction. 

We now show that $s^*=r^* {}^\frown (f_{\nu_0}^* \circ \psi^{-1},F^{**}(\nu_1)) \Vdash \varphi$. Suppose not. Let $s \leq s^*$ be such that $s \Vdash \neg \varphi$. Write $s=s_0 {}^\frown s_1 {}^\frown \vec{s}_2 {}^\frown \tp(s)$ where $s_0 \leq r^*$ and $s_0 {}^\frown s_1 \leq r^* {}^\frown (f_{\nu_0}^* \circ \psi^{-1}, F^{**}(\nu_1))$. Let $G$ be generic containing $s_0 {}^\frown s_1$, by the property of $i(r)$, let $s^\prime \in G$ be such that $s^\prime {}^\frown \tp(\bar{p}) \Vdash \varphi$, but then by extending $s^\prime$ (in $G$) if necessary, $s^\prime {}^\frown \vec{s}_2 {}^\frown \tp(s) \leq s, s^\prime {}^\frown \tp(\bar{p})$, so the condition forces both $\varphi$ and $\neg \varphi$, a contradiction.

Consider $t=r^* {}^\frown \tp(p^*+\langle \pi \rangle)$. Then the number of the block is $1$. We show that the condition forces $\varphi$. The point is for every extension of $t$ can be extended further to be an extension of $t+\langle \psi^\prime \rangle$, but since $\psi^\prime \in B_{r,0}$, then the condition will force $\varphi$. Hence, by a density argument, $t \Vdash \varphi$, but this contridicts the minimality of $n^*$.

\end{proof}

\end{proof}

Now we show that all cardinals are preserved. The forcing $P_\alpha$ is $\alpha^{++}$-c.c., so it preserves all cardinals greater than $\alpha^+$.
\begin{proposition}
\label{presbelow}
    For a cardinal $\beta<\alpha$ and a $P_\alpha$-name of a subset of $\beta$, $\Vdash_\alpha ``\dot{X} \in V^{P_\nu* \dot{P}_{\dot{\beta}_\nu/\nu}}"$ for some $\nu$. In particular, $P_\alpha$ preserves cardinals and cofinalities below $\alpha$.

\end{proposition}

\begin{proof}
    Let $p=\stem(p) {}^\frown \tp(p)$ where $\stem(p) \in Q$. Let $\dot{X}$ be a name of a subset of $\beta$. Find $p^* \leq^* p$ such that $\stem(p^*)=\stem(p)$ and for each $\gamma<\beta$, each $s \leq \stem(p)$, there is $s^* \leq^*s$ with $s^* {}^\frown \tp(p^*)$ decides $``\gamma \in \dot{X}"$. Let $\dot{X}^\prime$ be a $Q$-name such that if $G$ is a $Q$-generic, $\dot{X}^\prime[G]=\{\alpha \mid \exists s \in G(s {}^\frown \tp(p^*) \Vdash \alpha \in \dot{X})\}$. Then $p^* \Vdash \dot{X}=\dot{X}^\prime$.
\end{proof}

The forcing singularizes $\alpha$ to have cofinality $\omega$, and add $\alpha^{++}$ subsets of $\alpha$: for $\gamma \in [\alpha,\alpha^{++})$, define $t_\gamma: \omega \to \alpha$ as the following. By a density argument, let $p \in G$ be such that the common domain contains $\gamma$ and for $\mu$ appearing in $A^p$, $\gamma \in \dom(\gamma)$. Assume that $n^p$ is the number of the blocks in $p \setminus \tp(p)$. For $n>n^p$, find any $p^\gamma \in G$ such that the number of blocks in $p^\gamma \setminus \tp(p^\gamma) \geq n$. Write
$$p^\gamma= s_0 {}^\frown \cdots {}^\frown s_{n-2} {}^\frown (f_{n-1},s^\prime_{n-1}) {}^\frown \cdots {}^\frown (f_{k-1}, s^\prime_{n-1}) {}^\frown \langle f,\vec{f},A,F \rangle.$$
By compatibility between $p^\gamma$ and $p$, we have that $f(\gamma)$ has to be of the form $\check{\xi}_0$, $\xi_0 \in \dom(f_{n-1})$, $f_{n-1}(\xi_0)=\check{\xi}_1$, and so on. Define $t_\gamma(n)=f_{n-1} \circ \cdots \circ f_{k-1} \circ f(\gamma)$. 
Clearly $t_\alpha$ gives a cofinal sequence of $\alpha$ of length $\omega$, and hence, $\alpha$ is singularized to have cofinality $\omega$. Again, by a standard argument with the Prikry property, $\alpha^+$ is preserved. Since the forcing is $\alpha^{++}$-c.c., all the cardinals are preserved.
One can show that for $\gamma<\gamma^\prime$, there is $p \in G$ such that for every relevant object $\mu$ appearing in the tree part, $\gamma,\gamma^\prime \in \dom(\mu)$. Note that such $\mu$ is order-preserving. From here, use a density argument to show that $t_\gamma <^* t_{\gamma^\prime}$. Hence, the forcing violates the SCH at $\alpha$.

The set $C_\alpha$ is derived from the generic object as the following. If $G$ is $P_\alpha$-generic, define $C^\prime=\rng(t_\alpha)\cup\{\alpha\}$. Each condition $p \in G$ is of the form $$s {}^\frown \langle f,\vec{f},A,F\rangle$$ where $\nu_k=t_\alpha(k+1)$. In this case, the forcing $\dot{P}_{\dot{\xi}_k/\nu_k}$ also derives the set $C^k=C_{\xi_k/\nu_k} \subseteq [\nu_k,\xi_k)$, where $t_\alpha(k+1)=\nu_k<\xi_k<\nu_{k+1}=t_\alpha(k+2)$. Let $C_\alpha=C^\prime \cup \cup_{k<\omega} C^k$. Then $C_\alpha \subseteq \alpha+1$, $\max(C_\alpha)=\alpha$, $C_\alpha \setminus \{\alpha\}$ is a cofinal subset of $\alpha$, containing a subset of order-type $\omega$.
So far, we have verified items (\ref{1}) through (\ref{3}) of Proposition \ref{indsch}.

\begin{definition}[The quotient forcing]
Let $\dot{P}_{\alpha/\alpha}$ be the $P_\alpha$-name of the trivial forcing $(\{\emptyset\},\leq, \leq^*)$. In $V^{P_{\alpha}}$, let $\dot{C}_{\alpha/\alpha}$ be the $\dot{P}_{\alpha/\alpha}$-name of the empty set.
Now assume that $\beta<\alpha$.
Define $\dot{P}_{\alpha/\beta}$ as the following.
Let $G$ be $P_\beta$-generic. Define $\dot{P}_\alpha[G]=P_{\alpha/\beta}[G]$ as the forcing consisting of conditions of the form

$p=(\langle P_{\beta^\prime}[G],q^\prime) {}^\frown (\langle f_0 \rangle, \langle \dot{P}_{\dot{\beta}_0/\alpha_0}[G],\dot{q}_0\rangle)\cdots {}^\frown (\langle f_{n-1} \rangle, \langle \dot{P}_{\dot{\beta}_{n-1}/\alpha_{n-1}}[G],\dot{q}_{n-1}\rangle) {}^\frown \langle g_0,\vec{g},A,F \rangle$ where $n\geq 0$ and

\begin{enumerate}
    \item $\beta \leq \beta^\prime<\alpha$, so $P_{\beta^\prime}[G]$ was already defined by recursion, which is just $P_{\beta^\prime/\beta}[G]$, $q_0 \in P_{\beta^\prime}[G]$.
    \item If $n>0$, then $\alpha_0<\cdots<\alpha_{n-1}$, and for $i<n$, 
    \begin{itemize}
        \item let $d_i=\dom(f_i)$, then $d_i$ is an $\alpha_i$-domain, $d_i \in V$.
        \item for $\zeta \in d_0$, $\Vdash_{P_{\beta^\prime}[G]}``f_0(\zeta)<\alpha_0"$, and if $i>0$, then for $\zeta \in d_i$, $\Vdash_{P_{\alpha_{i-1}}[G]*\dot{P}_{\dot{\beta}_{i-1}/\alpha_{i-1}}[G]}``f_i(\zeta)<\alpha_i"$.
        \item $\Vdash_{P_{\alpha_i}[G]}``\alpha_i\leq \dot{\beta}_i<\alpha_{i+1}"$, where $\alpha_n=\alpha$.
        \item $\Vdash_{P_{\alpha_i}[G]}``\dot{q}_i \in \dot{P}_{\dot{\beta}_i/\alpha_i}[G]"$.
    \end{itemize}
     \item $A$ is a $E(d)$-tree.
    \item $d \in [\alpha^{++}]^{\leq \alpha}$, $d \in V$, is the {\em common domain} for $p$, i.e. $\dom(g_0)=d$, and $\vec{g}=\langle g_\nu \mid \nu \in A(\alpha)\rangle$ and for each $\nu$, $\dom(g_\nu)=d$.
    \item Fix $\zeta \in d$. If $n=0$, then $\Vdash_{P_{\beta^\prime}[G]}``g_0(\zeta)<\alpha"$, otherwise, $\Vdash_{P_{\alpha_{n-1}}[G]*P_{\dot{\beta}_{n-1}/\alpha_{n-1}}[G]} ``g_0(\zeta)<\alpha"$.
    \item for $\nu  \in A(\alpha)$ and $\zeta \in d$, $\Vdash_{P_\nu[G]*\dot{P}_{\dot{\beta}_\nu/\nu}[G]} ``g_\nu(\zeta)<\alpha"$.
    \item $\dom(F)=A(\alpha)$.
    \item for $\nu \in \dom(F)$, $F(\nu)=\langle \dot{P}_{\dot{\beta}_\nu/\nu}[G],\dot{q} \rangle$, where $\Vdash_{P_\nu[G]}``\nu \leq \dot{\beta}_\nu[G]<\alpha, \dot{q} \in \dot{P}_{\dot{\beta}_\nu/\nu}[G]"$
\end{enumerate}
Back in $V$. If $\dot{p} \in \dot{P}_{\alpha/\beta}$, then by density, the collection of $p_0 \in P_\beta$ such that $p_0$ decides $n$, $\alpha_0, \cdots, \alpha_{n-1}$, $\dom(f_0), \cdots, \dom(f_{n-1})$, the common domain, $A$, $q^\prime$ (as the equivalent $\dot{P}_{\dot{\beta}^\prime/\beta}$-name, and so on), is open dense. In this case, we say that $p_0$ {\em interprets} $\dot{p}$.
All in all, for such $p_0$ which interprets all the relevant components of $\dot{p}$, let $p_1$ be such the interpretation.
Write $p_0$ as $r_0 {}^\frown \langle g \rangle$ and by the interpretation, we may write $$p_1=(\langle \dot{P}_{\beta^\prime/\beta},\dot{q}^\prime) {}^\frown (\langle f_0 \rangle, \langle \dot{P}_{\dot{\beta}_0/\alpha_0},\dot{q}_0\rangle)\cdots {}^\frown (\langle f_{n-1} \rangle, \langle \dot{P}_{\dot{\beta}_{n-1}/\alpha_{n-1}},\dot{q}_{n-1}\rangle) {}^\frown \langle g_0, \vec{g},A,F \rangle.$$
There is a natural concatenation $p_0$ with $p_1$, written by $p_0{}^\frown p_1$, which is 
$$r=r_0 {}^\frown (\langle g \rangle, \langle \dot{P}_{\beta^\prime/\beta},\dot{q}^\prime \rangle) {}^\frown \cdots {}^\frown (\langle f_{n-1} \rangle, \langle \dot{P}_{\dot{\beta}_{n-1}/\alpha_{n-1}},\dot{q}_{n-1}\rangle) {}^\frown \langle g_0,\vec{g},A,F \rangle.$$
Then $r \in P_\alpha$ with $r \restriction P_\beta$ exists. We denote $p_1$ by $r/P_\beta$.
For $p_0$ and $p_1$ in $\dot{P}_{\alpha/\beta}$, we say that $p_0 \leq p_1$ if there is $p \in G^{P_\beta}$ such that $p$ interprets $p_0$ and $p_1$, and $p {}^\frown p_0 \leq_\alpha p {}^\frown p_1$. Also define $p_0 \leq^* p_1$ if there is $p \in G^{P_\beta}$ such that $p$ interprets $p_0$ and $p_1$, and $p {}^\frown p_0 \leq^*_\alpha p{}^\frown p_1$. One can check that the map $\phi: \{p \in P_\alpha \mid p \restriction P_\beta$ exists$\} \to P_\beta* \dot{P}_{\alpha/\beta}$ defined by
$\phi(p)=(p \restriction P_\beta, p / P_\beta)$
is a dense embedding, where $p \setminus P_\beta$ is the obvious component of $p$ which is in $\dot{P}_{\alpha/\beta}$.
Note that if $G$ is $P_\beta$-generic and $H$ is $P_\alpha[G]$-generic, there is a generic $I$ for $P_\alpha$ such that $V[G*H]=V[I]$, where $I$ is generated by $\{p \mid p \restriction P_\beta$ exists, $p \restriction P_\beta \in G$ and $(p /P_\beta)[G] \in H\}$. If $I$ is $P_\alpha$-generic and for some $ p\in I$, $p \restriction P_\beta$ exists, we can get $G$ which is $P_\beta$-generic and $H$ which is $P_\alpha[G]$-generic such that $V[G*H]=V[I]$ where $G$ is generated by $\{p \restriction P_\beta \mid p \in I$ and $p \restriction P_\beta$ exists$\}$ and $H=\{(p / P_\beta)[G] \mid p \in I$ and $p \restriction P_\beta$ exists$\}$.

In $V^{P_\beta}$, let $\dot{C}_{\alpha/\beta}$ be a $\dot{P}_{\alpha/\beta}$-name of the set described as the following. Let $G$ be $P_\beta$-generic.
and $H$ be generic over $P_\alpha[G]=\dot{P}_{\alpha/\beta}[G]$.
Then let $I=G*H$ be $P_\alpha$-generic. $I$ derives the set $C_\alpha \subseteq \alpha+1$ and $G$ derives the set $C_\beta \subseteq \beta+1$. Let $C_{\alpha/\beta}=C_\alpha \setminus C_\beta$.

\end{definition}

The following have the same proof as for $P_\alpha$ essentially. The one that we would like to point out is the closure property.

\begin{proposition}
\label{quotientprop}
    \begin{itemize}
        \item $\Vdash_\beta ``(\dot{P}_{\alpha/\beta},\leq^*)$ is $\alpha^{++}$-c.c."
        \item $\Vdash_\beta ``(\dot{P}_{\alpha/\beta},\leq, \leq^*)$ has the Prikry property.
        \item $\Vdash_\beta ``(\dot{P}_{\alpha/\beta},\leq^*)$ is $\beta^*$-closed", where $\beta^*$ is the least inaccessible cardinal greater than $\beta$.
    \end{itemize}

    \begin{proof}
        We only proof item $(3)$. For simplicity, let $\beta^\prime<\beta^*$ and in $V^{P_\beta}$, let $\langle p_\gamma \mid \gamma<\beta^\prime \rangle$ be a $\leq^*$-decreasing sequence.
        For simplicity, we consider the case where $p_\gamma=\langle P_\xi[G],q^\gamma \rangle {}^\frown \langle g^\gamma_0,\vec{g}^\gamma,A^\gamma,F^\gamma \rangle$ with the common domain $d^\gamma$.
        Since $(P_\xi[G],\leq^*)$ is $\beta^*$-closed, let $q^*$ be a $\leq^*$-lower bound of $q^\gamma$. In $V$, let $d^*=\cup \{d \mid \exists \gamma \exists p \in P_\beta(p \Vdash_\beta \dot{d}_\gamma=d)\}$. For all $\beta$ (including 0) with $g_\beta^\gamma$ exists, let $\dom(g_\beta^*)=d^*$, and for $\zeta \in d$, $g_\beta^*(\zeta)$ is forced to be the same as the interpretation $g_\beta^*(\zeta)$ for some sufficiently large $\gamma$, if exists, otherwise, $g_\beta^\gamma(\zeta)=\check{0}$. Let $A^*=\cap_{\gamma} \cap_{p} \{A \mid A \text{ is the pullback of } A^{\gamma,p}$ to the $d^*$-tree$\}$ where $p\Vdash_\beta ``\dot{A}^\gamma=A^{\gamma,p}"$. By shrinking, assume $\min(A^*(\alpha))>\beta$.
        Finally, for each $\gamma \in A^*(\alpha)$, the forcing which is relevant to $F^\gamma(\alpha)$ (for any $\gamma)$ is greater than $\gamma$-closed in the direct extension, and $\gamma>\beta$, so we can find $F^*$ such that $\langle P_\xi[G],q^*) {}^\frown \langle g^*,\vec{g}^*,A^*,F^* \rangle$ is a $\leq^*$-lower bound of $\langle p_\gamma \mid \gamma<\beta^\prime \rangle$. 
    \end{proof}
\end{proposition}

With all the definitions, one can verify the rest of Proposition \ref{indsch}.

\section{The general levels}
\label{generallevel}

Let $\alpha<\kappa$ be inaccessible. We may assume that $\alpha$ is greater than the first $\beta$ with $\circ(\beta)=1$. This forcing will generalize all of the forcings in previous sections.

\begin{definition}

A condition in $P_\alpha$ is of the form

$$p=\stem(p) {}^\frown \tp(p).$$

We have two cases.

\begin{enumerate}
    \item $\stem(p)$ is empty. In this case, $p$ is said to be {\em pure}.
    \item $\stem(p)$ is non-empty. In this case, $p$ is said to be {\em impure}. Then $\stem(p)$ is of the form 

    $$(s_0, \langle \dot{P}_{\dot{\beta}_0/\alpha_0},\dot{q}_0\rangle) {}^\frown \cdots {}^\frown (s_{n-1}, \langle \dot{P}_{\dot{\beta}_{n-1}/\alpha_{n-1}},\dot{q}_{n-1}\rangle),$$

    for some $n>0$. We say that the number of blocks in $\stem(p)$ is $n$. We have that
    \begin{itemize}
        \item $\alpha_0<\cdots<\alpha_{n-1}<\alpha$.
        \item for all $i$, $\Vdash_{\alpha_i} ``\alpha_i \leq \dot{\beta}_i<\alpha_{i+1}"$, where $\alpha_n=\alpha$.
        \item $(s_0, \langle \dot{P}_{\dot{\beta}_0/\alpha_0},\dot{q}_0\rangle) {}^\frown \cdots {}^\frown s_{n-1} \in P_{\alpha_{n-1}}$.
        \item $\Vdash_{\alpha_{n-1}} ``\dot{q}_{n-1} \in \dot{P}_{\dot{\beta}_{n-1}/\alpha_{n-1}}"$.
    \end{itemize}
    Equivalently, $\stem(p) \in P_{\alpha_{n-1}}*\dot{P}_{\dot{\beta}_{n-1}/\alpha_{n-1}}$.

\end{enumerate}

$\tp(p)$ also depends on $\stem(p)$ and $\alpha$. We have several cases.

\begin{enumerate}
    \item The case where $p$ is pure.
    \begin{enumerate}
        \item $\circ(\alpha)=0$. Then $\tp(p)=\langle f \rangle$, $f \in C(\alpha^+,\alpha^{++})$.
        \item $\circ(\alpha)>0$. In this case, $\tp(p)=\langle f_0,\vec{f},A,F \rangle$, where
        \begin{itemize}
            \item $A$ is a $d$-tree, with respect to $\vec{E}_\alpha(d)$.
            \item $\dom(F)=A(\alpha)$.
            \item for $\nu \in \dom(F)$, $F(\nu)=\langle \dot{P}_{\dot{\beta}_\nu/\nu},\dot{q} \rangle$ where $\Vdash_\nu ``\nu \leq \dot{\beta}_\nu<\alpha \text{ and } \dot{q} \in \dot{P}_{\dot{\beta}_\nu/\nu}"$.
            \item $\vec{f}=\langle f_\nu \mid \nu \in A(\alpha) \rangle$.
            \item there is a {\em common domain} $d$, which is an $\alpha$-domain, $\dom(f_0)=d$ and for all $\beta$, $\dom(f_\beta)=d$.
            \item $f \in C(\alpha^+,\alpha^{++})$ and for each $\nu \in A(\alpha)$, and $\zeta \in d$, $\Vdash_{P_\nu * \dot{P}_{\dot{\beta}_\nu/\nu}}``f_\nu(\zeta)<\alpha"$.
        \end{itemize}
        \end{enumerate}
        \item The case where $p$ is impure, say $\stem(p)\in P_{\alpha_{n-1}} * \dot{P}_{\dot{\beta}_{n-1}/\alpha_{n-1}}=:Q$.
        \begin{enumerate}
            \item $\circ(\alpha)=0$. Then $\tp(p)=\langle f \rangle$, $\dom(f)=d \in V$ is an $\alpha$-domain and for $\zeta \in d$, $\Vdash_Q ``f(\zeta)<\alpha"$.
            \item $\circ(\alpha)>0$. In this case, $\tp(p)=\langle f_0,\vec{f},A,F \rangle$, where there is a {\em common domain} $d \in [\alpha^{++}]^{\leq \alpha}$, $d \in V$, $d$ is an $\alpha$-domain such that
        \begin{itemize}
            \item $A$ is a $d$-tree, with respect to $\vec{E}_\alpha(d)$, $\min(A(\alpha))>\sup\{\gamma \mid \exists r \in P_{\alpha_{n-1}} (r \Vdash \dot{\beta}_{n-1}=\gamma)\}$.
            \item $\dom(F)=A(\alpha)$.
            \item for $\nu \in \dom(F)$, $F(\nu)=\langle \dot{P}_{\dot{\beta}_\nu/\nu},\dot{q} \rangle$ where $\Vdash_\nu ``\nu \leq \dot{\beta}_\nu<\alpha \text{ and } \dot{q} \in \dot{P}_{\dot{\beta}_\nu/\nu}"$.
            \item $\vec{f}=\langle f_\nu \mid \nu \in A(\alpha)\rangle$.
        
            \item $\dom(f_0)=d$ and for all $\nu$, $\dom(f_\nu)=d$.
            \item for $\zeta \in d$, $\Vdash_Q ``f_0(\zeta)<\alpha"$.
            \item for $\nu \in A(\alpha)$ and $\zeta \in d$, $\Vdash_{P_\nu*\dot{P}_{\dot{\beta}_\nu/\nu}}``f_\beta(\zeta)<\alpha"$.
        
        \end{itemize}
        \end{enumerate}
    
\end{enumerate}
\end{definition}

\begin{definition}[The one-step extension]
Assume $\circ(\alpha)>0$. Let $p=\stem(p) {}^\frown \langle f_0,\vec{f},A,F \rangle$ with the common domain $d$. Let $\langle \mu \rangle \in \Lev_0(A)$ with $\mu(\alpha)=\nu$. The {\em one-step extension of $p$ by $\mu$}, denoted by $p+\langle \mu \rangle$, is the condition $p^\prime=\stem(p^\prime) {}^\frown \langle g_0,\vec{g},A^\prime,F^\prime \rangle$ such that
\begin{enumerate}
    \item if $\circ(\nu)=0$, then $\stem(p^\prime)=\stem(p){}^\frown (f_0 \circ \mu^{-1},F(\nu))$, where $\dom(f_0 \circ \mu^{-1})=\rng(\mu)$, for $\gamma \in \dom(\mu)$, $f_0\circ \mu^{-1}(\mu(\gamma))=f_0(\gamma)$.

    \item if $\circ(\mu(\alpha))>0$, then $\stem(p^\prime)=\stem(p){}^\frown (\langle f_0 \circ \mu^{-1}, \langle f_\nu \circ \mu^{-1} \mid \nu \in (A \downarrow \mu)(\nu), A \downarrow \mu$, $F^\prime \rangle,F(\mu(\alpha)))$, where $\dom(F^\prime)=(A \downarrow \mu)(\nu)$, and for $\xi$, $F^\prime(\xi)=F(\xi)$(recall that $A \downarrow \mu=\{\vec{\tau} \circ \mu^{-1} \mid \vec{\tau}<\mu$ and for all $i$, $\circ(\tau_i(\alpha))<\circ(\mu(\alpha))$, so $(A \downarrow \mu)(\nu) \subseteq A(\alpha) \cap \nu)$.

    \item Write $Q$ as the forcing in which $\stem(p^\prime)$ lives. Say $Q=P_{\nu}*\dot{P}_{\dot{\beta}_\nu/\nu}$. Then
    \begin{itemize}
    \item $\Vdash_Q ``g_0=f_{\nu} \oplus \mu"$, namely $\dom(g_0)=d$, for $\zeta \in \dom(\mu)$, $g_0(\zeta)=\mu(\zeta)$, and for the other $\zeta$, $g_0(\zeta)=f_{\nu}(\zeta)$
    \item $\vec{g}=\{g_{\beta^\prime} \mid \beta^\prime \in \{\vec{\tau} \in A_{\langle \mu \rangle} \mid \tau_0(\alpha)>\xi^*\}$, where $\xi^*=\sup\{\gamma \mid \exists r \in P_\nu(r \Vdash_{\mu(\alpha)}\dot{\beta}_{\nu}=\gamma)\}$.
    \item $A^\prime=\{\vec{\tau} \in A_{\langle \mu \rangle} \mid \tau_0(\alpha)>\xi^*\}$.
    \item $F^\prime=F \restriction (A^\prime(\alpha))$.
    \end{itemize}
\end{enumerate}
We define $p+\langle \rangle$ as $p$, and by recursion, define $p+\langle \mu_0,\cdots, \mu_n \rangle=(p+\langle \mu_0,\cdots, \mu_{n-1} \rangle)+\langle \mu_n \rangle$.

\end{definition}

\begin{definition}[The direct extension relation]
Let $p=\stem(p) {}^\frown \tp(p)$ and $p^{\prime}=\stem(p^\prime) {}^\frown \tp(p^\prime)$. We say that $p$ is a {\em direct extension of} $p^\prime$, denoted by $p \leq^*_\alpha p^\prime$, if the following hold.

\begin{enumerate}
    \item $\stem(p) \leq^* \stem(p^\prime)$ (in some $Q:=P_{\alpha^\prime}*\dot{P}_{\dot{\beta}^\prime/\alpha^\prime})$.
    \item If $\circ(\alpha)=0$, write $\tp(p)=\langle f \rangle$ and $\tp(p^\prime)=\langle g \rangle$, then $\dom(f) \supseteq \dom(g)$, and for $\zeta \in \dom(g)$, $\Vdash_Q ``f(\zeta)=g(\zeta)"$.
    \item Suppose $\circ(\alpha)>0$. Write $\tp(p)=\langle f_0,\vec{f},A, F \rangle$ and $\tp(p^\prime)=\langle g_0,\vec{g},A^\prime,F^\prime \rangle$. Let $d^p$ and $d^{p^\prime}$ be the common domains for $p$ and $p^\prime$, respectively. Then
    \begin{itemize}
        \item $d^p \supseteq d^{p^\prime}$.
        \item $A \restriction d^{p^\prime} \subseteq A^\prime$.
        \item for $\zeta \in d^{p^\prime}$, $\Vdash_Q ``f_0(\zeta)=g_0(\zeta)"$.
        \item for $\nu \in A(\alpha)$ and $\vec{\mu} \in A$ with $\vec{\mu}(\alpha)=\nu$, say $F(\nu)=\langle \dot{P}_{\dot{\beta}_\nu/\nu},\dot{q} \rangle$, and for $\zeta \in d^{p^\prime}$, we have 
        $$p+\vec{\mu} \restriction (P_\nu * \dot{P}_{\dot{\beta}_\nu/\nu}) \Vdash_{P_\nu * \dot{P}_{\dot{\beta}_\nu/\nu}} ``f_\nu(\zeta)=g_\nu(\zeta)".$$
        \item for $\nu \in A(\alpha)$ and $\vec{\mu} \in A$ with $\vec{\mu}(\alpha)=\nu$,
        $$p+\vec{\mu} \restriction P_\nu \Vdash_\nu ``F(\nu)_0=F^\prime(\nu)_0 \text{ and } F(\nu)_1 \leq^*_{F(\nu)_0} F^\prime(\nu)_1"$$ 
    \end{itemize}
\end{enumerate}

\end{definition}

\begin{definition}[The extension relation]

Let $p=\stem(p) {}^\frown \tp(p)$ and $p^{\prime}=\stem(p^\prime) {}^\frown \tp(p^\prime)$. We say that $p$ is a {\em extension of} $p^\prime$, denoted by $p \leq_\alpha p^\prime$, if the following hold.

\begin{enumerate}
    \item The case $\circ(\alpha)=0$. Then
    \begin{itemize}
        \item $\stem(p) \leq \stem(p^\prime)$ in some $Q=P_{\alpha^\prime}* \dot{P}_{\dot{\beta}^\prime/\alpha^\prime}$.
        \item Write $\tp(p)=\langle f \rangle$ and $\tp(p^\prime)=\langle g\rangle$. Then $\dom(f) \supseteq \dom(g)$ and for $\zeta \in \dom(g)$, $\stem(p) \Vdash_Q ``f(\zeta)=g(\zeta)"$.
    \end{itemize}
    \item The case $\circ(\alpha)>0$. Then there is $\vec{\mu}$ (possibly empty) such that if $p^*=p^\prime+\vec{\mu}$, and we write $\tp(p)=\langle f,\vec{f},A,F \rangle$ and $\tp(p^*)=\langle g,\vec{g},A^*,F^* \rangle$, $d^p$ and $d^*$ are the common domains for $p$ and $p^*$, respectively, then
    \begin{itemize}
        \item $\stem(p) \leq \stem(p^*)$ in some $Q=P_{\alpha^\prime}* \dot{P}_{\dot{\beta}^\prime/\alpha^\prime}$.
       
        \item $d^p \supseteq d^{p^*}$.
        \item $A \restriction d^{p^*} \subseteq A^*$.
        \item for $\zeta \in d^{p^*}$, $\Vdash_Q ``f_0(\zeta)=g_0(\zeta)"$.
        \item for $\nu \in A(\alpha)$ and $\vec{\mu} \in A$ with $\vec{\mu}(\alpha)=\nu$, say $F(\nu)=\langle \dot{P}_{\dot{\beta}_\nu/\nu},\dot{q} \rangle$, and for $\zeta \in d^{p^\prime}$, we have 
        $$p+\vec{\mu} \restriction (P_\nu * \dot{P}_{\dot{\beta}_\nu/\nu}) \Vdash_{P_\nu * \dot{P}_{\dot{\beta}_\nu/\nu}} ``f_\nu(\zeta)=g_\nu(\zeta)".$$
        \item for $\nu \in A(\alpha)$ and $\vec{\mu} \in A$ with $\vec{\mu}(\alpha)=\nu$,
        $$p+\vec{\mu} \restriction P_\nu \Vdash_\nu ``F(\nu)_0=F^*(\nu)_0 \text{ and } F(\nu)_1 \leq^*_{F(\nu)_0} F^*(\nu)_1".$$
        (The last $\leq^*$ relation is intended).
    \end{itemize}

\end{enumerate}
Equivalently, $p \leq p^\prime$ if there is $\vec{\mu}$ such that $p$ is a condition obtained by extending the interleaving part of a direct extension of $p^\prime+\vec{\mu}$. We call $p^*$ the {\em interpolant} of $p$ and $p^\prime$. To be precise, $p^*$ is the unique condition such that $p^*=p+\vec{\mu}$ for some $\vec{\mu}$, $p^\prime$ is obtained by extending the interleaving part of a direct extension of $p^\prime$.
\end{definition}

\begin{proposition}
 $(P_\alpha,\leq)$ has the $\alpha^{++}$-chain condition.   
\end{proposition}
\begin{proof}
   Similar to the proof of Proposition \ref{cc1}.
\end{proof}

\begin{proposition}
    $(\{p \in P_\alpha \mid p \text{ is pure}\},\leq^*)$ is $\alpha$-closed.
\end{proposition}
\begin{proof}
    Similar to the proof of Proposition \ref{closure1}.
\end{proof}

\begin{theorem}
\label{Prikry2}
    $(P_\alpha,\leq,\leq^*)$ has the Prikry property, i.e. for $p \in P_\alpha$ and a forcing statement $\varphi$, there is $p^* \leq^* p$ such that $p^* \parallel \varphi$.
\end{theorem}

If $\circ(\alpha)=0$, any $p \in P_\alpha$ is a finite iteration of Prikry-type forcings, hence, it has the Pirkry property. The proof for $\circ(\alpha)=1$ is similar to the proof of Theorem \ref{prikry1}. We assume $\circ(\alpha)>1$. We need a few lemmas before we prove the Prikry property.

\begin{lemma}
\label{integrate}
Let $p \in P_\alpha$, $\beta<\circ(\alpha)$ with its common domain $d$, $\tp(p)=\langle f,\vec{f},A,F \rangle$. Fix $r \in P_{\nu^\prime} *\dot{P}_{\dot{\beta}_{\nu^\prime}/\nu^\prime}$. Let $\vec{\tau} \in A$ be unique such that there is $r^* \leq \stem(p+\vec{\tau})$ and $r$ is obtained by extending the interleaving part of $r$. Suppose is a measure-one set $B \in E_{\alpha,\beta}(d)$ such that 
\begin{enumerate}
\item for every $\nu \in B(\alpha)$, there is $\dot{q}_\nu^*$ such that $\Vdash_\nu ``\dot{q}_\nu^* \leq^* F(\nu)_1"$. Write $F^\prime(\nu)_1=\dot{q}_\nu^*$ for all $\nu$.
\item for every $\mu \in B$ with $\nu=\mu(\alpha)$, there are $r_\mu \leq^*r$, $f_\mu,\vec{f}_\mu,A_\mu,F_\mu$ such that 
\begin{align*}
    & r_\mu {}^\frown \langle f_\mu,\vec{f}_\mu,A_\mu,F_\mu,\langle \dot{P}_{\dot{\beta}_\nu/\nu},\dot{q}_\nu^*\rangle ) \leq^* \\  &r {}^\frown ( f_{\nu^\prime} \circ \mu^{-1},\langle f_\gamma \circ \mu^{-1} \mid \gamma \in (A_{\vec{\tau}} \downarrow \mu)(\nu) \rangle, A_{\vec{\tau}} \downarrow \mu, F^\prime \restriction (A_{\vec{\tau}} \downarrow \mu)(\nu),F(\nu)).
\end{align*}
\end{enumerate}
Then there is $p^* \leq^* p$ and $r^{**}$ such that 
\begin{itemize}
    \item for $\psi \in \Lev_0(A)$ with $\mu=\psi \restriction d$, $\circ(\mu(\alpha))=\beta$, we have that $r_\mu=r^{**}$,  $r^{**} {}^\frown (\tp(p^*)+\langle \psi \rangle) \leq^* r^{**} {}^\frown ( f_\mu,\vec{f}_\mu,A_\mu,F_\mu, \langle \dot{P}_{\dot{\beta}_\nu/\nu},\dot{q}_\nu^* \rangle)$.
    \item every extension of $r^{**} {}^\frown \tp(p^*)$ is compatible with $r^{**} {}^\frown (\tp(p^*)+\langle \psi \rangle)$ for some $\psi$ with $\circ(\psi(\alpha))=\beta$.
\end{itemize}
\end{lemma}
\begin{proof}
Assume for simplicity that $p$ is pure. First, we can shrink $B$ such that there is $r^*$, for all $\mu$, $r_\mu=r^*$. Then let
\begin{align*}
    f_{\nu^\prime}^*=j_{E_{\alpha,\beta}}(\mu \mapsto f_\mu)(\mc_{\alpha,\beta}(d)),\\
    \vec{f}^*=j_{E_{\alpha,\beta}}(\mu \mapsto \vec{f}_\mu)(\mc_{\alpha,\beta}(d)), \\
    A^*=j_{E_{\alpha,\beta}}(\mu \mapsto A_\mu)(\mc_{\alpha,\beta}(d)), \\
    F^*=j_{E_{\alpha,\beta}}(\mu \mapsto F_\mu)(\mc_{\alpha,\beta}(d)). 
\end{align*}
\begin{itemize}
    \item $f_{\nu^\prime}^*$ is forced to be an extension of $j_{E_{\alpha,\beta}}(\mu \mapsto f_{\nu^\prime} \circ \mu^{-1})(\mc_{\alpha,\beta}(d))=f_{\nu^\prime}$. Say $d^*=\dom(f_{\nu^\prime})$.
    \item by coherence, let $B_0$ generate $A^*$, we have that $B_0 \in \cap_{\gamma<\beta}E_{\alpha,\beta}(d^*)$, $A^* \subseteq A$.
    \item for each $\nu \in A(\alpha)$, $F^*(\nu)_1$ is forced to be a direct extension of $F(\nu)_1$.
    \item $\dom(\vec{f}^*)=A^*(\alpha)$.
    \item for $\nu \in A^*(\alpha)$, $f_\nu^*$ is forced to be a direct extension of $f_\nu$.
\end{itemize}
Let $f^*$ be $f \cup \{(\gamma,\check{0}) \mid \gamma \in d^* \setminus d\}$. Let $\mc=\mc_{\alpha,\beta}(d^*)$. Let $\mc=\mc_{\alpha,\beta}(d^*)=(j_{\alpha,\beta} \restriction d^*)^{-1}$. Then
\begin{itemize}
    \item $j_{\alpha,\beta}(f_{\nu^\prime}) \circ \mc=f_{\nu^\prime}$.
    \item $j_{\alpha,\beta}(A^*) \downarrow \mc=A^*$.
    \item Let $j_{\alpha,\beta}(\vec{f}^*)=\langle f^*_\gamma \rangle_\gamma$. Then $\langle f^*_\gamma \circ \mc^{-1} \mid \gamma \in A^*(\alpha) \rangle=\vec{f}^*$.
\end{itemize}
There is a measure-one set $B_1 \in E_{\alpha,\beta}(d^*)$ such that for $\psi \in B_1$,
\begin{enumerate}
    \item $\mu:=\psi \restriction d \in B$.
    \item $f_{\nu^\prime}^* \circ \psi^{-1}=f_\mu$.
    \item $A^* \downarrow \psi=A_\mu$.
    \item $F^*(\xi)=F_\mu(\xi)$ for $\xi \in A_\mu(\mu(\alpha))$.
    \item $\langle f^*_\gamma \circ \psi^{-1} \mid \gamma \in A_\mu \rangle=\vec{f}_\mu$.
\end{enumerate}
For $\nu \in B_1(\alpha)$, let $f_\nu^*=f_\nu \cup \{(\gamma,\check{0}) \mid \gamma \in d^* \setminus d\}$ and $F^*(\nu)_1=F^\prime(\nu)$. Finally, let $B_2$ be the collection of $d$-object $\psi$ with $\circ(\psi(\alpha))>\beta$ and $B_1 \downarrow \psi \in E_{\psi(\alpha),\beta}(\psi[d^* \cap \dom(\psi)])$.
For $\nu \in B_2$, let $f_\nu^*=f_\nu \cup \{(\gamma,\check{0}) \mid \gamma \in d^* \setminus d\}$ and $F^*(\nu)=F^\prime(\nu)$. Let $A^{**}$ be generated by $B_0 \cup B_1 \cup B_2$, and $p^*=\langle f_0^*,\vec{f}^*,A^{**},F^* \rangle$. To show that $p^*$ satisfies the properties, note that for $\psi \in \Lev_0(A)$ with $\circ(\psi(\alpha))=\beta$, $\psi \in B_1$, and by the property of $B_1$, the first requirement for $p^*$ is straightforward. To show the predense property, let $s \leq p^*$ such that there is an initial segment of $s$, $r_0$, which is an extension of $r^{**}$. Let $\vec{\tau} \in A^{**}$ be unique such that there is $s^\prime\leq^*p^*+\vec{\tau}$ and $s$ is obtained by extending the interleaving part of $s^\prime$. If $\vec{\tau}=\emptyset$, then pick any $\psi \in \Lev_0(A^s)$ such that $\circ(\psi(\alpha))=\beta$, then $s+\langle \psi \rangle \leq r^{**} {}^\frown  p^*+\langle \psi \restriction d^* \rangle$. If $\vec{\tau}=\neg \emptyset$. If for all $i$, $\circ(\tau_i(\alpha))<\beta$, then take any $\psi$ as before with $\vec{\tau}<\psi$. Since $\psi \in B_1$, then the key point is that $\vec{\tau} \in A_\mu$. From here, we can show that $s \leq s^{**} {}^\frown p^*+\langle \tau \rangle$. Suppose now that $i$ is the least such that $\circ(\tau_i(\alpha)) \geq \beta$. If $\circ(\tau_i(\alpha))=\beta$, then as before, $s \leq r^{**} {}^\frown p^* {}^\frown \langle \tau_i \rangle$. If $\circ(\tau_i(\alpha))>\beta$, let $\mu=\tau_i \restriction d$, and let $\psi^\prime \in A_\mu$ which appears in the corresponding measure-one set when we add $\tau_i$. Then $\psi^\prime=\psi \circ \tau_i^{-1}$ for some $i$. We can then show that $s+\langle \psi^\prime \rangle {}^\frown r^{**} {}^\frown p^*+\langle \psi \rangle$. This completes the proof.
\end{proof}

\begin{lemma}
\label{maximize2}
Let $p \in P_\alpha$ and $\varphi$ be a forcing statement. Then there is $p^* \leq^* p$ such that if $r=r_0 {}^\frown \tp(r)$, $r \leq p^*$, $p^\prime$ is the interpolant of $r$ and $p^*$, and $r \parallel \varphi$, then

$$r_0  {}^\frown \tp(p^\prime) \parallel \varphi \text{ the same way}.$$
\end{lemma}

\begin{proof}
    The proof is essentially the same as the proof of Lemma \ref{maximize1}.
\end{proof}

\begin{lemma}
Let $p$ be a condition and $\varphi$ be a forcing statement. Then there is $p^* \leq^* p$, $\tp(p^*)=\langle f^*,\vec{f}^*,A^*,F^* \rangle$ such that for every object $\mu$ which appears in $A$, say $\nu=\mu(\alpha)$, we have that for every $r \in P_{\nu^\prime}* \dot{P}_{\dot{\beta}_{\nu^\prime}/\nu^\prime}$ for some $\nu^\prime<\nu$, there is $r^* \leq r^*$ and a unique $i \in \{0,1,2\}$ such that
\end{lemma}

\begin{proof}[Proof of Theorem \ref{Prikry2}]
Let $p$ be a condition and $\varphi$ be a forcing statement. Assume $p$ is pure and satisfies Lemma \ref{maximize2}. Write $p=\langle f,\vec{f},A,F \rangle$ with its common domain $d$.
We will build a $\leq^*$-decreasing sequence $\langle p^\gamma \mid \gamma<\alpha \rangle$ below $p$, and write $p^\gamma=\langle f^\gamma, \vec{f}^\gamma, A^\gamma,F^\gamma \rangle$, where the common domain is $d^\gamma$, such that
    \item for $\nu<\alpha$, $|\{\gamma \mid \nu \in A^\gamma(\alpha)\}| \leq \eta$.
In the end, we take $p^*=\langle f^*,\vec{f}^*,A^*,F^* \rangle$ such that $f^*=\cup f^\gamma$, $A^*=\Delta_\nu A^\nu$, for $\eta \in A^*$, $F^*(\eta)_1$ is a $\leq^*$-lower bound of $\{F^\nu(\eta)\}_\nu$ (possible since the number of $\nu$ such that $F^\nu(\alpha)$ exists is small), and $\Vdash_{P_\eta*\dot{P}_{\dot{\beta}_\eta/\eta}} ``f^*_\eta=\cup f^\nu_\eta"$. Then $p^* \leq^* p$ and then we show that $p^*$ satisfies the Prikry property. Let $\langle r_\gamma \mid \gamma< \alpha \rangle$ be an enumeration of $r \in  \cup_{\nu<\alpha} P_\nu *\dot{P}_{\dot{\beta}_\nu/\nu}$ such that there is $\vec{\tau} \in A$, $r \leq \stem(p+\vec{\tau})$, and we let $\vec{\tau}_\gamma$ be unique such that there is $r^* \leq^* p+\vec{\tau}_\gamma$, $r_\gamma$ is obtained by extending the interleaving part of $r^*$ only.

For $\gamma$ limit, take $p^\gamma$ as a $\leq^*$-lower bound of $\langle p^{\gamma^\prime}\mid \gamma^\prime<\gamma \rangle$. Suppose $p^\gamma$ is constructed, we now consstruct $p^{\gamma+1}$. Assume $r_\gamma \in P_{\nu^\prime}* \dot{P}_{\dot{\beta}_{\nu^\prime}/\nu^\prime}$. Fix $\nu>\nu^\prime$. By Remark \ref{ppoitre}, assume that $A^\gamma$ is generated by $B^\gamma \subseteq \mathcal{B}_{d^\gamma}$. Let $\{\mu_\xi \mid \xi<\nu^{++}\}$ be the collection of $\mu \in B^\gamma$ with $\mu(\alpha)=\nu$. Let $\dot{G}_\nu$ be the canonical name for $P_\nu* \dot{P}_{\dot{\beta}_\nu/\nu}$-generic. By the Prikry property, let $r_{\gamma,\xi} \leq^* r_\gamma$, $f^{\gamma,\xi}$, $A^{\gamma,\xi}$, $F^{\gamma,\xi}$, $\dot{q}^*_\nu$, and $\vec{f}^{\gamma,\xi}$, such that 
\begin{align*}
    & r_{\gamma,\xi} {}^\frown ( f^{\gamma,\xi}, \vec{f}^{\gamma,\xi},A^{\gamma,\xi},F^{\gamma,\xi}, \langle \dot{P}_{\dot{\beta}_\nu/\nu},\dot{q}_\nu^* \rangle) \leq^* \\
    & r_\gamma {}^\frown ( f^\gamma_{\nu^\prime} \circ \mu_\xi^{-1},\langle f^\gamma_\eta \mid \eta \in (A^\gamma_{\vec{\tau}_\gamma} \downarrow \mu_\xi)(\nu) \rangle,A^\gamma_{\vec{\tau}_\gamma} \downarrow \mu_\xi,F^\gamma \restriction (A^\gamma_{\vec{\tau}_\gamma} \downarrow \mu_\xi)(\nu),F^\gamma(\nu)),  
\end{align*}
and $r_{\gamma,\xi} {}^\frown ( f^{\gamma,\xi}, \vec{f}^{\gamma,\xi},A^{\gamma,\xi},F^{\gamma,\xi}, \langle \dot{P}_{\dot{\beta}_\nu/\nu},\dot{q}_\nu^* \rangle)$ decides $$\varphi_{\vec{\tau}_\gamma,\mu_\xi} \equiv \exists t \in \dot{G}_\nu (t{}^\frown (\tp(p+(\vec{\tau}_\gamma {}^\frown \langle \mu_\xi \rangle) \restriction d)).$$
Notice that $\dot{q}_\nu^*$ does not depend on $\xi$. This can be done since $\Vdash_\nu ``(\dot{P}_{\dot{\beta}_\nu/\nu},\leq^*)$ is $\nu^{+3}$-closed" (it is much higher than $\nu^{+3}$-closed). 
By extending further regarding the direct extension relation, assume that $r_{\gamma,\xi} {}^\frown ( f^{\gamma,\xi}, \vec{f}^{\gamma,\xi},A^{\gamma,\xi},F^{\gamma,\xi}, \langle \dot{P}_{\dot{\beta}_\nu/\nu},\dot{q}_\nu^* \rangle) \Vdash \varphi_{\vec{\tau}_\gamma,\mu_\xi}^i$ for a unique $i \in \{0,1,2\}$, where
\begin{align*}
    \varphi_{\vec{\tau}_\xi,\mu_\xi}^0 &\equiv \exists t \in \dot{G}_\nu(t {}^\frown \tp(p+\vec{\tau}_\xi {}^\frown \langle \mu_\xi \rangle) \Vdash \varphi), \\
    \varphi_{\vec{\tau}_\xi,\mu_\xi}^1 & \equiv \exists t \in \dot{G}_\nu(t {}^\frown \tp(p+\vec{\tau}_\xi {}^\frown \langle \mu_\xi \rangle) \Vdash \neg \varphi),\\
     \varphi_{\vec{\tau}_\xi,\mu_\xi}^2 & \equiv \nexists t \in \dot{G}_\nu(t {}^\frown \tp(p+\vec{\tau}_\xi {}^\frown \langle \mu_\xi \rangle) \parallel  \varphi).
\end{align*}
We now change some notations to ease us at the end of the proof. For each $\mu=\mu_\xi$, and $\nu=\mu(\alpha)$, let $fr_{\gamma,\xi}=r_\mu$, $f_{\nu^\prime}^\mu=f^{\gamma,\xi}$, for each $\eta$, $f_\eta^\mu=f_\eta^{\gamma,\xi}$, $A^\mu=A^{\gamma,\xi}$, and $F^\mu=F^{\gamma,\xi}$. 
Then let $i^\gamma_\mu$ be the unique $i$ such that $r_{\gamma,\xi} {}^\frown ( f^{\gamma,\xi}, \vec{f}^{\gamma,\xi},A^{\gamma,\xi},F^{\gamma,\xi}, \langle \dot{P}_{\dot{\beta}_\nu/\nu},\dot{q}_\nu^* \rangle) \Vdash \varphi_{\vec{\tau}_\xi,\mu}^{i^\gamma_\mu}$. For $\beta<\circ(\alpha)$, let $B_{\gamma,\beta}^i=\{\mu \mid \circ(\mu(\alpha))=\beta$ and $ i^\gamma_\mu=i\}$. There is unique $i_{\gamma,\beta}$ such that $B_{\gamma,\beta}^{i_{\gamma,\beta}} \in E_{\alpha,\beta}(d^\gamma)$. We now consider two cases.

Case 1: for all $\beta<\circ(\alpha)$, $i_{\gamma,\beta}=2$. Let $A^{\gamma+1}$ be a $d^\gamma$-tree generated by $\cup_{\beta<\circ(\alpha)} B_{\gamma,\beta}^2$, for $\nu \in A^{\gamma+1}(\alpha)$, $F^{\gamma+1}(\nu)=\langle \dot{P}_{\dot{\beta}_\nu/\nu},\dot{q}_\nu^* \rangle$ and $p^{\gamma+1}=\langle f^\gamma,\vec{f}^\gamma, A^{\gamma+1},F^{\gamma+1} \rangle$.

Case 2: there is $\beta<\circ(\alpha)$, $i_{\gamma,\beta}<2$. For each $\mu$, we have that for each $\mu \in B_{\gamma,\beta}^{i_{\gamma,\beta}}$, say $\mu=\mu_\xi$, we have $r_\mu=r_{\gamma,\xi}$, $f_\mu=f^{\gamma,\xi}$, $\vec{f}_\mu=\vec{f}^{\gamma,\xi}$, $A_\mu=A^{\gamma,\xi}$, and $F_\mu=F^{\gamma,\xi}$. Apply Lemma \ref{integrate} to obtain $p^{\gamma+1} \leq^* p^\gamma$.

We now finish the construction. 

\begin{claim}
$p^*$ satisfies the Prikry property.
\end{claim}
\begin{proof}
    If there is $p^{**} \leq^* p^*$ which decides $\varphi$, then we also finish. Suppose not. Let $p^{**} \leq^* p^*$ with the minimal number of blocks of $p^{**}$ such that $p^{**} \parallel \varphi$. Without loss of generality, assume $p^{**} \Vdash \varphi$. We demonstrate the case where $n^{p^{**}}=2$. Assume that $p^{**}$ is of the form
    $$(g_0,\vec{g}_0,A_0,F_0,\langle \dot{P}_{\dot{\beta}_0/\alpha_0},\dot{q}_0\rangle) {}^\frown (g_1,\vec{g}_1,A_1,F_1,\langle \dot{P}_{\dot{\beta}_1/\alpha_1},\dot{q}_1\rangle) {}^\frown \langle h,\vec{h},T,H \rangle. $$
    Let $s\leq^* p^*+\langle \psi_0,\psi_1 \rangle$ be such that $p^{**}$ is obtained by extending only the interleaving part of $s$. Let $\mu_0=\psi_0 \restriction d$ and $\mu_1=\psi_1 \restriction d$. Then $r:=(g_0,\vec{g}_0,A_0,F_0,\langle \dot{P}_{\dot{\beta}_0,\alpha_0},\dot{q}_0 \rangle) \leq \stem(p+\langle \mu_0 \rangle)$. Hence, $r=r_\gamma$ for some $\gamma$. We now consider the construction of $p^{\gamma+1}$. Let $\nu^\prime=\alpha_0$, $\nu=\alpha_1$, and $\psi \restriction d^\gamma=\mu$. From the notation of Case 2 in the construction of $p^{\gamma+1}$, we have that $r_\mu=r_{\gamma,\xi} \leq^* r$, and 
    $$r_\mu {}^\frown (g_1,\vec{g}_1,A_1,F_1, \dot{P}_{\dot{\beta}_1/\alpha_1},\dot{q}_1\rangle) \leq^* r_\mu {}^\frown (f_\mu,\vec{f}_\mu,A_\mu,F_\mu,\dot{P}_{\dot{\beta}_1/\alpha_1},\dot{q}_\nu^*\rangle).$$
    We claim that there is $\beta<\circ(\alpha)$ with $i_{\gamma,\beta}=0$. First, note that $\stem(p^{**}) \leq p+\langle \mu_0,\mu_1 \rangle$. By Lemma \ref{maximize2},
    $\stem(p^{**}) {}^\frown \tp(p+\langle \mu_0,\mu_1 \rangle) \Vdash \varphi$. Thus, it is not possible for $i_{\gamma,\circ(\nu)}$ to be $0$ (otherwise, we can choose a generic $G_\nu$ containing $\stem(p^{**})$, and this will give a contradiction since $\stem(p^{**}) {}^\frown \tp(p+\langle \mu_0,\mu_1 \rangle \Vdash \varphi$). Hence, there is $\beta<\circ(\alpha)$ such that $i_{\gamma,\beta}<2$ and we chose a measure-one set from $E_{\alpha,\beta}(d^\gamma)$ to integrate and construct $p^{\gamma+1}$. If $\circ(\nu)=\beta$, then clearly $i_{\gamma,\beta}=0$. Suppose $\circ(\nu) \neq \beta$.

    Case 1: $\circ(\nu)<\beta$: Choose $\psi_2 \in A^{p^{**}}$ such that $\circ(\psi_2(\alpha))=\beta$, $\nu_2=\psi_2(\alpha)$, and write $\mu_2= \psi_2 \restriction d^\gamma$. By the construction of $p^{\gamma+1}$, we have that $p^{**}+\langle \psi_2 \rangle \leq p^*+\langle \psi_0,\psi_2 \rangle$.
    Choose $G$ that contains $\stem(p^{**}+\langle \psi_2 \rangle)$, then $G$ contains $r {}^\frown ( f_{\mu_2},\vec{f}_{\mu_2},A_{\mu_2}, F_{\mu_2}, \langle \dot{P}_{\dot{\beta}_{\nu_2}/\nu_2},\dot{q}_{\nu_2}^* \rangle)$. Since $i_{\gamma,\beta}=0$ or $1$, there is $t \in G$ such that $t {}^\frown \tp(p+\langle \mu_0,\mu_1,\mu_2 \restriction d \rangle) \Vdash \varphi$, but since $\stem(p^{**}) {}^\frown \tp(p+\langle \mu_0,\mu_1 \rangle) \Vdash \varphi$, by the choice of $G$, we have $i_{\gamma,\beta}=0$.

    Case 2: $\circ(\nu)>\beta$: then choose $\psi_2 \in \Lev_0(A_1)$ such that $\circ(\psi_2(\alpha))=\beta$. Consider $p^{**}+\langle \psi_2 \rangle$. Note that $\psi_2 \restriction d^\nu=\psi \circ \psi_1^{-1}$ for some $\psi \in B_{\gamma,\beta}^{i_{\gamma,\beta}}$. Then use a similar argument as in Case 1 to show that $i_{\gamma,\beta}=0$.

We now conclude that $i_{\gamma,\beta}=0$, Recall the notion of $r^{**}$ from Lemma \ref{integrate}. A similar argument on the choice of genericity shows that for every $\psi \in \Lev_0(A^*_{\langle \psi_0 \rangle})$ with $\mu=\psi \restriction d$, we have $r^* {}^\frown \langle f_\mu,\vec{f}_\mu,A_\mu,F_\mu,(\dot{P}_{\dot{\beta}_{\mu(\alpha)}/\mu(\alpha)},\dot{q}_{\mu(\alpha)}^*\rangle) {}^\frown \tp(p^*+\langle \psi_0,\psi \rangle) \Vdash \varphi$. By a density argument and a predense property stated as in Lemma \ref{integrate}, we then have that $r^* {}^\frown \tp(p^*+\langle \psi_0 \rangle) \Vdash \varphi$. This contradicts the minimality of $n^{p^{**}}$, and this finishes the proof.
\end{proof}
\end{proof}
We will now consider the cardinal arithmetic, and a preservation of cardinals and cofinalities together.

\begin{proposition}
\label{presbel2}
    For a cardinal $\beta<\alpha$ and a $P_\alpha$-name of a subset of $\beta$, $\Vdash_\alpha \dot{X} \in V^{P_\nu* \dot{P}_{\dot{\beta}_\nu/\nu}}$. As a consequence, $P_\alpha$ preserves all cardinals, and $\alpha$ is preserved.
\end{proposition}

\begin{proof}
Similar to Proposition \ref{presbelow}.
\end{proof}

Note that unlike the forcing at the level of the first measurable cardinal, $P_\alpha$ may singularize cardinals below $\alpha$.
Since $P_\alpha$ has the $\alpha^{++}$-chain condition, all cardinals from $\alpha^{++}$ and above are preserved. 

We now derive the club $C_\alpha$ from $P_\alpha$. For generality, we consider the case $\circ(\alpha)>0$. Let $G$ be $P_\alpha$-generic. Then for each $\nu<\alpha$ such that by letting $Q_\nu=P_\nu * \dot{P}_{\dot{\beta}_\nu/\nu}$, we have that $G \restriction Q_\nu$ exists. $G  \restriction Q_\nu$ is $Q_\nu$-generic, and it introduces a set $C^\nu \cup C^{\beta_\nu/\nu}$ where $\beta_\nu=\dot{\beta}_\nu[G \restriction P_\nu]$, $C^\nu \subseteq \nu+1$ with $\max(C^\nu)=$, $C^{\beta_\nu/\nu} \subseteq (\nu,\beta_\nu]$ such that $\max(C^{\beta_\nu/\nu})=\beta_\nu$ if $\beta_\nu>\nu$, otherwise, $C^{\beta_\nu/\nu}=\emptyset$.
Let $C_\alpha=(\cup_{\{\nu \mid G \restriction Q_\nu \text{ exists}\}} (C^\nu \cup C^{\beta_\nu/\nu})) \cup \{\alpha\}$. Since $\circ(\alpha)>0$, we can perform one-step extension of any condition so that $\{\nu \mid G \restriction Q_\nu$ exists$\}$ is unbounded in $\alpha$. Like in the extender-based Magidor-Radin forcing, one can induct $\{\nu \mid Q_\nu$ exists$\}$ has a tail of order-type $\omega^{\circ(\alpha)}$. Hence, in $V[G]$, $\alpha$ is singularized to have cofinality $\cf(\omega^{\circ(\alpha)})$. From here and the Prikry property, one can show that $\alpha^+$ is preserved. We conclude that all cardinals are preserved. Also, note that for $\nu<\nu^\prime$ such that $G \restriction Q_\nu$, $G \restriction Q_{\nu^\prime}$ exist, we have that $C^\nu \cup C^{\beta_\nu/\nu}$ is an initial segment of $C^{\nu^\prime}$, so it is an initial segment of $C_\alpha$. Thus, $\lim(C_\alpha)=(\cup_{\{\nu \mid G \restriction Q_\nu \text{ exists}\}} (\lim(C^\nu) \cup \lim(C^{\beta_\nu/\nu}))) \cup \{\alpha\}$. As in Proposition \ref{presbel2}, the cardinal arithmetic of cardinals below $\alpha$ are determined at levels below $\alpha$. For $\xi<\alpha$, we have that by Proposition \ref{indsch} items $(\ref{3})$ and $(\ref{6})$,  either $2^\xi=\xi^+$ or $2^\xi=\xi^{++}$, and $2^\xi=\xi^{++}$ iff $\xi \in \lim(C^\nu) \cup \lim(C^{{\beta}_\nu/\nu})$.
Hence, the cardinal arithmetic below $\alpha$ satisfies $(\ref{3})$ of Proposition \ref{indsch}. Since $\alpha \in \lim(C_\alpha)$, it remains to show that $2^\alpha=\alpha^{++}$.

Work with a pure condition $p \in G$. Enumerate $\{\nu \mid G \restriction P_\nu$ exists$\}$ increasingly as $\{\nu_i \mid i<\omega^{\cf(\alpha)}\}$.  Fix $\gamma \in [\alpha,\alpha^{++})$. By a density argument, let $p^\gamma \leq p$, $p^\gamma \in G$ be such that if $\tp(p^\gamma)=\langle f^\gamma ,\vec{f}^\gamma,A^\gamma,F^\gamma \rangle$, then for every object $\mu$ which appears in $A^\gamma$, $\gamma \in \dom(\mu)$. Suppose that $\stem(p^\gamma) \in P_{\nu_{i_\gamma}}* \dot{P}_{\dot{\beta}_{\nu_{i_\gamma}}/\nu_{i_\gamma}}$.
For $i \leq i_\gamma$, define $t_\gamma(i)=0$. For $i>i_\gamma$, there is an extension $p^{\gamma,i} \in G$ such that 
\begin{enumerate}
    \item $p^{\gamma,i} \restriction P_{\nu_{i_\gamma}}$ exists.
    \item by writing $p^{\gamma_i}$ as $$(s_0, \langle \dot{P}_{\dot{\beta}_0/\alpha_0},\dot{q}_0\rangle) {}^\frown \cdots {}^\frown (s_{n-1}, \langle \dot{P}_{\dot{\beta}_{n-1}/\alpha_{n-1}},\dot{q}_{n-1}\rangle) {}^\frown \langle f,\vec{f},A,F \rangle,$$
    then $(s_0, \langle \dot{P}_{\dot{\beta}_0/\alpha_0},\dot{q}_0\rangle) {}^\frown \cdots {}^\frown s_k \in P_{\nu_{i_\gamma}}$, and
    \begin{itemize}
        \item $f(\gamma)$ is a check-name $\check{\gamma}_0$, then $\gamma_0 \in f_{n-1}$, where $f_{n-1}$ is the first coordinate of $s_{n-1}$.
        \item by recursion, $\gamma_0, \cdots, \gamma_{l-1}$ is defined for $l<n-k-1$, then $\gamma_{l-1} \in \dom(f_{n-l})$, where $f_{n-l}$ is the first coordinate of $s_{n-l}$, and $f_{n-l}(\gamma_{l-1})$ is a check-name $\gamma_l$.
    \end{itemize}
\end{enumerate}
We define $t_\gamma(i)=f_k(\gamma_{n-k-1})$.
For $\gamma<\gamma^\prime$, there is a condition $p^{\gamma,\gamma^\prime} \in G$ such that if $A^{\gamma,\gamma^\prime}$ is the tree appearing in $\tp(p^{\gamma,\gamma^\prime}$, we have that for every $\mu$ appearing in $A^{\gamma,\gamma^\prime}$, $\gamma,\gamma^\prime \in \dom(\mu)$ and $\mu(\gamma)<\mu(\gamma^\prime)$. From this, it can be shown that $t_\gamma <^* t_{\gamma^\prime}$, which means there is $i^*$ such that for $i>i^*$, $t_\gamma(i)<t_{\gamma^\prime}(i)$. This gives $\alpha^{++}$ different functions from $\omega^{\cf(\alpha)}$ to $\alpha$. It is easy to show that $\alpha$ is a strong limit cardinal, and so in $V[G]$, $2^\alpha=\alpha^{\cf(\alpha)} \geq \alpha^{++}$. Since $P_\alpha$ is $\alpha^{++}$-c.c., $2^\alpha=\alpha^{++}$ as desired.
So far we have show items $(\ref{1})$, $(\ref{2})$, and $(\ref{3})$ of Proposition \ref{indsch}. It remains to consider the business regarding the quotient forcings.

\begin{definition}[The quotient forcing]

Let $\dot{P}_{\alpha/\alpha}$ be the $P_\alpha$-name of the trivial forcing $(\{\emptyset\},\leq, \leq^*)$. In $V^{P_{\alpha}}$, let $\dot{C}_{\alpha/\alpha}$ be the $\dot{P}_{\alpha/\alpha}$-name of the empty set.
Now assume that $\beta<\alpha$.
Define $\dot{P}_{\alpha/\beta}$ as the following.
Let $G$ be $P_\beta$-generic. Define $P_\alpha[G]=\dot{P}_{\alpha/\beta}[G]$ as the forcing consisting of conditions of the form $\stem(p) {}^\frown \tp (p)$, where

\begin{enumerate}
    \item $\stem(p)$ is of the form
    $$ (P_{\beta^\prime}[G],q^\prime) {}^\frown (s_0,\langle \dot{P}_{\beta_0/\alpha_0}[G],\dot{q}_0) {}^\frown (s_{n-1},\langle \dot{P}_{\dot{\beta}_{n-1}/\alpha_{n-1}}[G],\dot{q}_{n-1} \rangle),$$
    for some $n$ (if $n=0$, then $\stem(p)$ is only $(P_{\beta^\prime}[G],q^\prime)$) such that
    \begin{itemize}
    \item $P_{\beta^\prime}[G]=\dot{P}_{\dot{\beta}^\prime/\beta}[G]$, and $q^\prime \in P_{\beta^\prime}[G]$.
    \item if $n>0$, then $\alpha_0<\cdots<\alpha_{n-1}$, and for $i<n$,
    \begin{itemize}
        \item if $\circ(\alpha_i)=0$, $s_i=\langle f_i \rangle$, and if $\circ(\alpha_i)>0$, $s_i=\langle f_i,\vec{f}_i,A_i,F_i \rangle$, where $d_i=\dom(f_i)$ is an $\alpha_i$-domain, $d_i \in V$.
        \item for $\zeta \in d_0$, $\Vdash_{P_{\beta^\prime}[G]} ``f_0(\zeta)<\alpha_0"$ and if $i>0$, then for $\zeta \in d_i$, $\Vdash_{P_{\alpha_{i-1}}[G]*\dot{P}_{\dot{\beta}_{i-1}/\alpha_{i-1}}[G]}``f_i(\zeta)<\alpha_i"$.
        \item $\Vdash_{P_{\alpha_i}[G]} ``\alpha_i \leq \dot{\beta}_i<\alpha_{i+1}"$, where $\alpha_n=\alpha$.
        \item $\Vdash_{P_{\alpha_i}[G]}``\dot{q}_i \in \dot{P}_{\dot{\beta}_i/\alpha_i}[G]"$.
        \item if $\circ(\alpha_i)>0$,
        \begin{itemize}
            \item $A_i$ is a $d_i$-tree with respects to $\vec{E}_{\alpha_i}(d_i)$ (in the sense of $V$).
            \item $\vec{f}_i=\langle f_{i,\nu} \mid \nu \in A_i(\alpha_i) \rangle$.
            \item for each $\nu$, $\dom(f_{i,\nu})=d_i$.
            \item for $\zeta \in d_i$, $\Vdash_{P_\nu[G] * \dot{P}_{\dot{\beta}_\nu/\nu}[G]}``f_{i,\nu}(\zeta)<\alpha_i"$.
            \item $\dom(F_i)=A_i(\alpha_i)$.
            \item for $\nu \in A_i(\alpha_i)$, $F_i(\nu)=\langle \dot{P}_{\dot{\beta}_\nu/\nu}[G],\dot{q} \rangle$, $\Vdash_{P_\nu[G]} ``\nu \leq \dot{\beta}_\nu<\alpha_i"$ and $\Vdash_{P_\nu[G]}``\dot{q} \in \dot{P}_{\dot{\beta}_\nu/\nu}[G]"$.
        \end{itemize}
    \end{itemize}
    \end{itemize}
    \item if $\circ(\alpha)=0$, then $\tp(p)$ is $\langle f \rangle$, and if $\circ(\alpha)>0$, then $\tp(p)=\langle f,\vec{f},A,F \rangle$, where there is a {\em common domain} $d$, which is an $\alpha$-domain (in the sense of $V$) such that
    \begin{itemize}
        \item If $\circ(\alpha)=0$, then $\dom(f)=d$ and for $\zeta \in d$, $\Vdash_{P_{\alpha_{n-1}}[G]*\dot{P}_{\dot{\beta}_{n-1}/\alpha_{n-1}}[G]} ``f(\zeta)<\alpha"$.
        \item Assume $\circ(\alpha)>0$. Then,
        \begin{itemize}
            \item $A$ is a $d$-tree with respects to $\vec{E}_\alpha(d)$ (in the sense of $V$).
            \item $\dom(F)=d$ and for $\nu \in \dom(F)$, $F(\nu)=\langle \dot{P}_{\dot{\beta}_\nu/\nu}[G],\dot{q} \rangle$ where $\Vdash_{P_\nu[G]}``\nu \leq \dot{\beta}_\nu<\alpha$ and $\dot{q} \in \dot{P}_{\dot{\beta}_\nu/\nu}[G]"$.
        \item $\dom(f)=d$, $\vec{f}=\langle f_\nu \mid \nu \in A(\alpha) \rangle$, and for all $\nu$, $\dom(f_\nu)=d$.
        \item for $\zeta \in d$, $\Vdash_{P_{\alpha_{n-1}}[G]*\dot{P}_{\dot{\beta}_{n-1}/\alpha_{n-1}}[G]} ``f(\zeta)<\alpha"$ and for $\nu \in A(\alpha)$, $\Vdash_{P_\nu[G]*\dot{P}_{\dot{\beta}_\nu/\nu}[G]} ``f_\nu(\zeta)<\alpha"$.

    \end{itemize}
    \end{itemize}
\end{enumerate}
Back in $V$. If $\dot{p} \in \dot{P}_{\alpha/\beta}$, then by density, the collection of $p_0 \in P_\beta$ such that $p_0$ decides $n$, $\alpha_0, \cdots, \alpha_{n-1}$, $\dom(f_0), \cdots, \dom(f_{n-1})$, the common domain, $A_i$, $A$, $q^\prime$ (as the equivalent $\dot{P}_{\dot{\beta}^\prime/\beta}$-name, and so on), is open dense. In this case, we say that $p_0$ {\em interprets} $\dot{p}$.
All in all, for such $p_0$ which interprets all the relevant components of $\dot{p}$, let $p_1$ be such the interpretation.
Assume $\circ(\beta)>0$ and $\circ(\alpha)>0$ (the other cases are simpler) write $p_0$ as $r_0 {}^\frown \langle g,\vec{g},B,H \rangle$ and by the interpretation, we may write $$p_1=(\langle \dot{P}_{\beta^\prime/\beta},\dot{q}^\prime) {}^\frown (s_0, \langle \dot{P}_{\dot{\beta}_0/\alpha_0},\dot{q}_0\rangle)\cdots {}^\frown (s_{n-1}, \langle \dot{P}_{\dot{\beta}_{n-1}/\alpha_{n-1}},\dot{q}_{n-1}\rangle) {}^\frown \langle f,\vec{f},A,F \rangle.$$
There is a natural concatenation $p_0$ with $p_1$, written by $p_0{}^\frown p_1$, which is 
$$r=r_0 {}^\frown (\langle g,\vec{g},B,H\rangle, \langle \dot{P}_{\beta^\prime/\beta},\dot{q}^\prime \rangle) {}^\frown \cdots {}^\frown (s_{n-1}, \langle \dot{P}_{\dot{\beta}_{n-1}/\alpha_{n-1}},\dot{q}_{n-1}\rangle) {}^\frown \langle f,\vec{f},A,F \rangle.$$
Then $r \in P_\alpha$ with $r \restriction P_\beta=p_0$ exists. Denote $r/P_\beta$ the term $p_1$.
For $P_\beta$-names $p_0$ and $p_1$ in $\dot{P}_{\alpha/\beta}$, we say that $p_0 \leq p_1$ if there is $p \in G^{P_\beta}$ such that $p$ interprets $p_0$ and $p_1$, and $p {}^\frown p_0 \leq_\alpha p {}^\frown p_1$. Also define $p_0 \leq^* p_1$ if there is $p \in G^{P_\beta}$ such that $p$ interprets $p_0$ and $p_1$, and $p {}^\frown p_0 \leq^*_\alpha p{}^\frown p_1$. One can check that the map $\phi: \{p \in P_\alpha \mid p \restriction P_\beta$ exists$\} \to P_\beta* \dot{P}_{\alpha/\beta}$ defined by
$\phi(p)=(p \restriction P_\beta, p / P_\beta)$
is a dense embedding, where $p \setminus P_\beta$ is the obvious component of $p$ which is in $\dot{P}_{\alpha/\beta}$.
Note that if $G$ is $P_\beta$-generic and $H$ is $P_\alpha[G]$-generic, there is a generic $I$ for $P_\alpha$ such that $V[G*H]=V[I]$, where $I$ is generated by $\{p \mid p \restriction P_\beta$ exists, $p \restriction P_\beta \in G$ and $(p /P_\beta)[G] \in H\}$. Conversely, if $I$ is $P_\alpha$-generic and for some $ p\in I$, $p \restriction P_\beta$ exists, we can get $G$ which is $P_\beta$-generic and $H$ which is $P_\alpha[G]$-generic such that $V[G*H]=V[I]$, where $G$ is generated by $\{p \restriction P_\beta \mid p \in I$ and $p \restriction P_\beta$ exists$\}$ and $H=\{(p /P_\beta)[G] \mid p \in I$ and $p \restriction P_\beta$ exists$\}$. 

In $V^{P_\beta}$, let $\dot{C}_{\alpha/\beta}$ be a $\dot{P}_{\alpha/\beta}$-name of the set described as the following. Let $G$ be $P_\beta$-generic.
and $H$ be generic over $P_\alpha[G]=\dot{P}_{\alpha/\beta}[G]$.
Then let $I=G*H$ be $P_\alpha$-generic. $I$ derives the set $C_\alpha \subseteq \alpha+1$ and $G$ derives the set $C_\beta \subseteq \beta+1$. Let $C_{\alpha/\beta}=C_\alpha \setminus C_\beta$.

\end{definition}

The following proposition has a similar proof as some previous propositions, for example, Proposition \ref{cc1} and Proposition \ref{quotientprop}.
\begin{proposition}
    \begin{itemize}
        \item $\Vdash_\beta ``(\dot{P}_{\alpha/\beta},\leq)$ is $\alpha^{++}$-c.c."
        \item $\Vdash_\beta ``\dot{P}_{\alpha/\beta},\leq, \leq^*)$ has the Prikry property.
        \item $\Vdash_\beta ``(\dot{P}_{\alpha/\beta},\leq^*)$ is $\beta^*$-closed", where $\beta^*$ is the least inaccessible cardinal greater than $\beta$.
    \end{itemize}
\end{proposition}    

We conclude that from all the analysis, Proposition \ref{indsch} holds for $P_\alpha$ and all relevant quotients at $\alpha$.    
\section{The main forcing}
\label{mainforcing}

We are now defining our main forcing $\mathbb{P}$. The forcing $\mathbb{P}=\cup_{\{\alpha<\kappa \mid \alpha \text{ is inaccessible}\}} P_\alpha$. For $p$ and $p^\prime$ in $\mathbb{P}$, define $p \leq p^\prime$ if $p \in P_\alpha$, $p^\prime \in P_{\alpha^\prime}$, $\alpha \geq \alpha^\prime$, $p \restriction P_{\alpha^\prime}$ exists, and $p \restriction P_{\alpha^\prime} \leq_{\alpha^\prime} p$.
The forcing is $\kappa^+$-c.c.
Let $G$ be $\mathbb{P}$-generic. Then if $p \in G$ is such that $p \restriction P_\alpha$ exists, then $G \restriction P_\alpha$ is $P_\alpha$-generic. We briefly describe $\mathbb{P}/P_\alpha$ for $\alpha<\kappa$ inaccessible. Recall that for $\alpha\leq \eta<\kappa$, $\Vdash_\alpha ``\{p/P_\alpha \mid p \in P_\eta, p \restriction P_\alpha \text{ exists}\}$ is densely embedded in $\dot{P}_{\eta/\alpha}"$. 
For $\alpha<\kappa$ inaccessible, let $\mathbb{P}/P_\alpha$ as the collection $\{p/P_\alpha \mid p \in \mathbb{P}, p \restriction P_\alpha \text{ exists}\}$. Define $p_0 \leq p_1$ (in $V^{P_\alpha})$ if there is $p \in P_\alpha$ such that $p {}^\frown p_0 \leq_{\mathbb{P}} p^{}\frown p_1$. 

\begin{remark}
    $V^{P_\alpha}$, for every $p \in \mathbb{P}/P_\alpha$, there is $\eta$ such that $p \in \dot{P}_{\eta/\alpha}$.
\end{remark}

 This introduces the set $C_\alpha$. Let $C=\cup_\alpha \{C_\alpha \mid G \restriction P_\alpha$ is $P_\alpha$-generic$\}$. Then $C \subseteq \kappa$ is a club. The next theorem shows that the cardinal arithmetic should be as expected.

\begin{theorem}
\label{deciding}
Let $\dot{f}$ be a $\mathbb{P}$-name of a function from $\beta$ to ordinals such that $\beta<\kappa$ and $G$ is $\mathbb{P}$-generic. Then $f \in V[G \restriction P_\alpha]$ for some $\alpha<\kappa$.   
\end{theorem}

\begin{proof}

We show by a density argument. Let $p \in \mathbb{P}$ and $\dot{f}$ be a $\mathbb{P}$-name of functions from $\beta$ to ordinals, where $\beta<\kappa$. For simplicity, assume $p$ is an empty condition.
Let $M \prec H_\theta$ for some sufficiently large regular $\theta$, $\beta \subseteq M$, $\dot{f},p,\mathbb{P} \in M$, $V_{M \cap \kappa} \subseteq M$, and $\circ(M \cap \kappa) \geq \beta$ (this is possible from Assumption \ref{iniass}, item $(4)$). Say $\alpha=M \cap \kappa$. We are going to build $p^* \in P_\alpha$ of the form $p^*=\langle f,\vec{f},A,F \rangle$. Let $f,\vec{f}$, and $A$ be any objects. Fix $\gamma<\beta$ and $\nu \in A(\alpha)$ such that $\circ(\nu)=\gamma$.
Let $Y_\nu$ be a maximal antichain of relevant collections in $P_\nu$.
For each $r \in Y_\nu$, let $G_r$ be $P_\nu$-generic containing $r$. Since $V_\alpha \subseteq M$, $M[G] \cap \kappa=M \cap \kappa$. Find $q \in \mathbb{P}/G$ such that $q$ decides $\dot{f}(\gamma)[G]$. By elementarity, we may find such a $q$ in $M[G]$. Then $q \in P_\xi/G$ for some $\xi<\alpha$.
Back in $M$, let $\dot{\xi}$ and $\dot{q}$ be the names for such $\xi$ and $q$. Define $F(\nu)=\langle \dot{P}_{\dot{\xi}/\nu},\dot{q} \rangle$. For $\nu$ with $\circ(\nu) \geq \beta$, we assign $F(\nu)$ to be any value. This completes the construction of $F$. By our design, we have that $p^*$ decides $\dot{f}$, and hence, $p^* \Vdash_{\mathbb{P}} \dot{f} \in V^{P_\alpha}$.

\end{proof}

\begin{corollary}
    Every cardinal is preserved in $V^{\mathbb{P}}$.
\end{corollary}

\begin{corollary}
    For $\beta<\kappa$ the value $2^\beta$ is determined in $V^{\mathbb{P}_\alpha}$ for some $\alpha \in (\beta,\kappa)$. 
\end{corollary}

\begin{corollary}
$\kappa$ is inaccessible in $V^\mathbb{P}$.
\end{corollary}

\begin{proof}
    By Theorem \ref{deciding}, if $\kappa$ is collapsed, then the witness function has to be in $V^{P_\alpha}$ for some $\alpha<\kappa$, but $\kappa$ is preserved in $P_\alpha$, a contradiction. The same argument shows that $\kappa$ is regular. Finally, for every $\beta<\kappa$, the value $2^\beta$ must be determined in $V^{P_\alpha}$ for some sufficiently large $\alpha$ because the forcing can be factored so that the quotient forcing after the stage $\beta$ is $\beta^+$-closed under the direct extension,
\end{proof}

\begin{theorem}
    In $V^\mathbb{P}$, $\kappa$ is inaccessible, there is a club $D \subseteq \kappa$ such that for $\alpha\in D$, $2^\alpha=\alpha^{++}$ and for $\alpha \not \in D$, $2^\alpha=\alpha^+$.
\end{theorem}

\begin{proof}
    Let $C$ be the club derived from $\mathbb{P}$ and $D=\lim(C)$. Then $D$ satisfies the theorem.
\end{proof}

\section{Getting different cardinal behaviors on stationary classes}
\label{getdifcar}

Assume GCH. Let $\kappa$ be a strongly inaccessible cardinal. For each $\gamma<\kappa$, let $f_\gamma:\kappa \to \kappa$. Assume that for each $\gamma$, there is a coherent sequence of extenders $\vec{E}_\gamma$, on a  set $X_\gamma \subseteq \kappa$ and $\circ^\gamma:X_\gamma \to \kappa$ such that

\begin{itemize}
\item $\vec{E}_\gamma=\langle E_\gamma(\alpha,\beta) \mid \beta<\circ^\gamma(\alpha) \rangle$.
\item each $E_\gamma(\alpha,\beta)$ is an $(\alpha,\alpha^{+f_\gamma(\alpha)})$ extender witnesses $\alpha$ being $\alpha^{+f_\gamma(\alpha)}$-strong.
\item $\circ^\gamma(\alpha)<\alpha$.
\item for $\nu<\kappa$, $\{\alpha \mid \circ^\gamma(\alpha) \geq \nu\}$ is stationary.
\end{itemize}
 Then we can proceed a similar forcing construction, except that the corresponding Cohen part at $\alpha$ will be $C(\alpha^+,\alpha^{f_\gamma(\alpha)})$. Let $\mathbb{P}^{\langle \vec{f}_\gamma \mid \gamma<\kappa \rangle}$ be the corresponded forcing.

\begin{theorem}
\label{carsta}
    In the forcing $\mathbb{P}^{\langle f_\gamma \mid \gamma<\kappa \rangle}$, all the cardinals are preserved, the forcing produces a club $C \subseteq \cup_{\gamma<\kappa} X_\gamma$ such that for each $0<\xi<\kappa$ regular and $\gamma<\kappa$, the collection of $\alpha$ with $\cf(\alpha) \geq \xi$ and $2^\alpha=\alpha^{+f_\gamma(\alpha)}$ is stationary. 
\end{theorem}

\begin{proof}[Proof Sketch]
Fix $\xi>0$ and a $\mathbb{P}$-name of a club subset of $\kappa$ $\dot{D}$. 
Let $p$ be a condition, $\dot{D}$ a name of a club subset of $\kappa$. Let $M \prec H_\theta$ where $\theta$ is sufficiently large, $\dot{D},p,\mathbb{P}^{\langle f_\beta \mid \beta<\kappa \rangle} \in M$, $V_{M \cap \kappa} \subseteq M$, and $\circ^\gamma(M \cap \kappa)\geq \xi$. Let $\alpha=M \cap \kappa$. We are now extending $p$ to a condition whose top level is $\alpha$. Let $p=\langle f,\vec{f},A,F \rangle \in P_\alpha$, where $f,\vec{f},A$ can be any sensible components. For each $\nu \in A(\alpha)$, let $F(\nu)$ be a condition that decides an element $\dot{\xi}$ which is the minimum of the interpretation of $\dot{D} \setminus (\nu+1)$. By elementarity, $\dot{\xi}$ is decided to be below $\alpha$. Then the final condition forces that $\alpha$ is in $\dot{C} \cap \dot{D}$, and forces that $2^\alpha=\alpha^{f_\gamma(\alpha)}$, and $\cf(\alpha) \geq \xi$.
\end{proof}

\begin{example}
Start from GCH, $\kappa$ carrying a $(\kappa,\kappa^{+\kappa})$-extender. Then it is possible that for $\gamma<\kappa$, there is a sequence coherent sequence of extenders $\vec{E}_\gamma$ on a stationary set $X_\gamma \subseteq \kappa$ where each $E_\gamma(\alpha,\beta)$ witnesses $\alpha$ being $\alpha^{+\gamma}$-strong. Let $f_\gamma$ be a constant function with value $\gamma$. Then the forcing $\mathbb{P}^{\langle f_\gamma \mid \gamma<\kappa \rangle}$ forces that $\kappa$ is inaccessible, and in $V_\kappa$ and each $\gamma<\kappa$, there is a stationary class $S_\gamma \subseteq \kappa$ such that for $\alpha \in S_\gamma$, $2^\alpha=\alpha^{+\gamma}$. Also, in this situation, for each $\gamma$ and $\xi$, the collection of $\alpha$ such that $\circ^\gamma(\alpha)=\xi$ is stationary. A similar proof as in Theorem \ref{carsta} shows that in the final model, for every $\gamma<\kappa$ and $\xi<\kappa$ regular, there is a stationary set of $\alpha$ such that $2^\alpha=\alpha^{+\gamma}$ and $\cf(\alpha)=\xi$.
\end{example}

\bibliographystyle{ieeetr}
\bibliography{references}
\end{document}